\documentclass[a4paper,12pt]{article}

\usepackage{amsmath,amssymb,amsfonts,amsthm,enumitem,xcolor,mathrsfs,tikz-cd}
\usepackage[top=2cm,left=2cm,right=2cm,bottom=2cm]{geometry}
\usepackage[hidelinks]{hyperref}
\usepackage{tikz}
\usetikzlibrary{calc}

\usepackage[style=alphabetic,maxnames=50]{biblatex}
\bibliography{bibliography}

\title{Generically hereditarily equivalent continua and topological characterization of generic maximal chains of generalized Ważewski dendrites.}

\newcommand{\cont}{\operatorname{Cont}}
    \newcommand{\comp}{\operatorname{Comp}}

    \newcommand{\moa}{\operatorname{MOA}}
    \newcommand{\full}{\operatorname{Full}}
    
    \newcommand{\bb}[1]{\mathbb{#1}}

    \newcommand{\scr}[1]{\mathscr{#1}}

    \newcommand{\ord}{\operatorname{ord}}
    \newcommand{\lr}[1]{\langle #1 \rangle}
    \newcommand{\ol}[1]{\overline{#1}}
    \newcommand{\ep}{\operatorname{End}}
    \newcommand{\Root}{\operatorname{Root}}
    \newcommand{\id}{\operatorname{id}}

\theoremstyle{definition}
\newtheorem{theo}{Theorem}[section]
\newtheorem{lem}[theo]{Lemma}
\newtheorem{defi}[theo]{Definition}
\newtheorem{prop}[theo]{Proposition}
\newtheorem{cor}[theo]{Corollary}

\newtheorem{nota}[theo]{Notation}
\newtheorem{question}[theo]{Question}

\newtheoremstyle{break}
  {\topsep}{\topsep}%
  {\itshape}{}%
  {\bfseries}{}%
  {\newline}{}%
\theoremstyle{break}

\author{Bryant Rosado Silva\footnote{https://orcid.org/0009-0009-3710-9599}
\footnote{Supported by the grant GACR 24-10705S and by Charles University Research Centre program No. UNCE/24/SCI/022}\\
Department of Mathematical Analysis\\
Faculty of Mathematics and Physics, Charles University\\
Prague, Czechia\\
E-mail: silva@karlin.mff.cuni.cz
\and
Benjamin Vejnar\footnote{https://orcid.org/0000-0002-2833-5385} \footnote{
Supported by the grant GACR 24-10705S
}\\
Department of Mathematical Analysis\\
Faculty of Mathematics and Physics, Charles University\\
Prague, Czechia\\
E-mail: vejnar@karlin.mff.cuni.cz}

\begin{document}

\maketitle
\begin{abstract}
The notion of hereditarily equivalent continua is classical in continuum theory with only two known nondegenerate examples (arc, and pseudo-arc).
In  this paper we introduce generically hereditarily equivalent continua, i.e. continua which are homeomorphic to comeager many subcontinua. We investigate this notion in the realm of Peano continua and we prove that all the generalized Wa\.zewski dendrites are such.
Consequently, we study maximal chains consisting of subcontinua of generalized Wa\.zewski dendrites and we prove that there is always a comeager orbit under the natural action of the homeomorphism group. As a part of the proof, we provide a topological characterization of the generic maximal chain.
\end{abstract}

\renewcommand{\thefootnote}{}



\section{Introduction}

In 1921 Mazurkiewicz proposes the question whether the arc is the only space $X$ with the property that every nondegenerate subcontinuum of $X$ is homeomorphic to $X$. The answer to this question is negative, with one example being found in different ways through the years (the pseudo-arc), thus a new terminology appeared: a space is a hereditarily equivalent continuum if it is homeomorphic to each of its nondegenerate subcontinua. For sake of simplicity we will abbreviate it as HEC. Are there any other HEC except the arc and the pseudo-arc? This question remains open until the present moment of writing of this article. The most recent advance that we are aware of is the publication \cite{classification-hoehn-oversteegen}. If one looks to the hyperspace of a continuum $X$, it is HEC if and only if
$$\mathcal{A}=\{K \in \cont(X) \ | \ K \simeq X\} = \cont(X) \setminus \operatorname{Fin}_1(X)$$
where $\operatorname{Fin}_1=\{\{x\} \ | \ x \in X\}$ and $\cont(X)=\{K \subseteq X  \ | \ K \text{ is a continuum }\}$. Both are considered with the Vietoris topology.
It is well known that $\operatorname{Fin}_1(X)$ is a closed set with empty interior, thus, $\mathcal{A}$ is an open dense subset of $\cont(X)$, hence a comeager subset. This leads to the introduction of the following definition:

\begin{defi}{}
 A continuum $X$ is called \textbf{generically hereditarily equivalent} if
 $$\{K \in \cont(X) \ | \ K \simeq X\}$$
 is a comeager subset of $\cont(X)$. In this case, we say that $X$ is GHEC.
\end{defi}

Then, two related questions comes to mind. First, is there an example of a space that is GHEC but not HEC? Second, can we find all the spaces that are GHEC?  Note that the second question is stronger than classifying all the spaces that are HEC.

This paper revolves around three main ideas, which are reflected in its title. First, we show that there exists not only one but a class of spaces that are GHEC but not HEC. These examples are the generalized Wa\.zewski dendrites. Moreover, we have found additional restrictions that the collection of homeomorphic copies can satisfy but still provide a comeager subset (see Theorem \ref{generalizeddendrite}). We also prove that every two elements of this collection are ambiently homeomorphic (see Theorem \ref{stronglyWM}).

We approach the second question in the case of Peano continua, closing the possibilities only to dendrites with arcwise densely many branching points or the arc (see Corollary \ref{PeanoGHEC}).

\begin{question}
Is every Peano GHEC a generalized Wa\.zewski dendrite or an arc?
\end{question}

Second, the GHEC condition is extended to the hyperspace of maximal order arcs yielding two other properties: GCGHEC when comeager many chains are made of comeager many copies of the space and GCHEC when comeager many chains are made of copies of the space apart from the starting point (see Definition \ref{defgenericchain}.) Then, the relation between these new properties and being GHEC is explored, with hereditary indecomposability allowing to connect them.

\begin{question}
    Are there other topological properties that link the GHEC, GCHEC, and GCGHEC conditions?
\end{question}

Finally, the largest part of this article is dedicated to the study of these maximal order arcs of the generalized Wa\.zewski dendrites. In Theorem \ref{ambientlyeqWM} we give a topological characterization of the generic maximal order arc of a generalized Wa\.zewski dendrite. This yields the following corollary:

\begin{cor}{}\label{CorComeagerOrbit}
 If $M \subseteq \{3,4,\ldots\} \cup \{\omega\}$, then the action $\operatorname{Homeo}(W_M) \curvearrowright \moa(W_M)$ given by
 $$h(\mathcal{C}) = \{h(K) \in \cont(W_M) \ | \ K \in \mathcal{C}\}$$
 has a comeager orbit.
\end{cor}

With a similar argumentation, we have also shown that there is another class of chains of the Wa\.zewski dendrite that are ambiently equivalent (see Theorem \ref{ambientlyeqW}).

\begin{question}
    Is there any space for which GCGHEC holds but not GCHEC?
\end{question}

Let us briefly discuss the motivation and importance of our results from the point of dynamical systems. It is known (due to Ellis) that for every topological group $G$ there is a minimal (continuous) action of $G$ on a compact space $M(G)$ such that any minimal action of $G$ on a compact space is a factor of $M(G)$. The action is called the universal minimal flow and the space $M(G)$ is called the universal minimal space of $G$. Recently, a lot of attention was paid to the questions, whether for a given group $G$, the space $M(G)$ is a singleton or whether $M(G)$ is metrizable. This question was especially considered when $G$ is the group of automorphisms of some very symmetric structures (see e.g. \cite{KechrisPestovTodorcevic}) or the group of homeomorphisms of very homogeneous spaces like the circle, the Cantor set, the Lelek fan, generalized Wa\.zewski dendrites \cite{BartosovaKwiatkowska, GlasnerWeiss,Kwiatkowska, Pestov}.
In some of these cases, the collection of maximal chains turned out to be an important tool for the description of the universal minimal flow.

While studying the universal minimal flow of the group $\operatorname{Homeo}(M)$, where $M$ is a closed manifold of dimension at least three, the authors of  \cite{GutmanTsankovZucker}
have shown that the maximal connected chains of $M$ have meager orbits. This is a completely different behavior than what we get in Corollary \ref{CorComeagerOrbit}.
Our positive result also contrasts with that of
\cite{basso2024surfaces}
where it is proved that closed manifolds of dimension at least two which are different from the sphere and the projective plane 
as well as one-dimensional Peano continua without local cut points 
do not have a generic chain.


\section{Preliminaries}

A metrizable space $X$ is Polish if it is a separable completely metrizable space. A subset of $A \subseteq X$ is $G_\delta$ if $A=\cap_{n \in \bb{N}}U_n$ where $U_n$ is an open set for every $n \in \bb{N}$. In completely metrizable spaces, the Baire Category Theorem states that this intersection is dense if and only if each $U_n$ is dense. A set $B \subseteq X$ is said to be nowhere dense if its closure has an empty interior. It is meager if it is the countable union of nowhere dense sets and comeager if its complement is meager. If a property $\mathcal{P}(x)$ is such that it holds for comeager many points, we write $\forall^\ast x \in X : \mathcal{P}(x)$.

A continuum is a nonempty, compact, connected, and metrizable topological space while a Peano continuum is a continuum that is locally connected. An arc in a topological space $X$ is a homeomorphic image of $[0,1] \subseteq \bb{R}$. An arc $A$ is free if its interior is an open set in $X$. It is known that Peano continua are arcwise connected, that is, between every pair of points $x,y$ there exists an arc connecting $x$ and $y$, so we say that $x$ and $y$ are the endpoints of the arc. A Peano continuum $X$ is indecomposable if there are not two proper subcontinua $A,B \subseteq X$ such that $X = A \cup B$. It is hereditarily indecomposable if every subcontinuum is indecomposable. A connected topological space $X$ is unicoherent if whenever $X=A \cup B$ where $A$ and $B$ are connected closed subsets of $X$, it holds that $A \cap B$ is connected. We say that $X$ is hereditarily unicoherent if every closed, connected subspace is unicoherent. We recall the following theorem:

\begin{theo}[Boundary Bumping Theorem - {\cite[p. 73]{nadlercontinuum}}]
Let $X$ be a continuum and let $U$ be a proper, nonempty, open subset of $X$. If $K$ is a component of $\overline{U}$, then $K \cap U^c \ne \emptyset$.
\end{theo}

A dendrite $X$, is a Peano continuum without simple closed curves, thus it is uniquely arcwise connected. Hence, given $x,y \in X$ we denote the arc connecting $x$ and $y$ by $[x,y]$. The set of all branching points of $X$ is denoted by $\mathcal{B}(X)$. Let $X,Y$ be dendrites with $X \subseteq Y$. Given that dendrites are uniquely arcwise connected, one can define the first point map $r_{Y,X}: Y \to X$ as follows: for each $y \in Y$, $r_{Y,X}(y)$ is the unique point in $X$ that lies in every arc from $y$ to a point of $X$. We say that the order of a point $x$ of a dendrite $X$ is $\kappa=\operatorname{ord}(x,X)$ for some cardinal $\kappa$ if for any open neighborhood $U$ of $x$ there exists an open set $V$ with $V \subseteq U$ having boundary of cardinality smaller than $\kappa$, that is, $|\partial{V}|<\kappa$. If a point $x \in X$ has order one, it is an endpoint, if it is greater than two, it is called a branching point. The set of endpoints is denoted by $\ep(X)$ while the set of branching points is denoted by $\mathcal{B}(X)$. If $Y$ is a subdendrite of $X$ and $y \in Y$, we say that ther order of $y$ in $Y$ is maximal in $X$ if $\operatorname{ord}(y,Y)= \operatorname{ord}(y,X)$. On \cite{charatonikdendrites} one can find precious information about dendrites compiled.

Throughout this text $M$ is a nonempty subset of $\{3,4,5,\ldots,\omega\}$. The generalized Wa\.zewski dendrite $W_M$ is a dendrite whose branching points of order $m$ are arcwise dense  for every $m \in M$ (i.e. the interior of any arc contains densely many branching points of order $m$ for every $m \in M$.)

In a metric space $(X,d)$ we denote by $B_d(x,r)$ (or just $B(x,r)$ when there is only one metric being used) the open ball of radius $r$ centered in $x$. If $(X,d)$ is a metric space and $A,B \subseteq X$ are nonempty compact subsets of $X$, the distance from $A$ to $B$ is given by
$$\overline{d}(A,B) = \max_{a \in A} \min_{b \in B} d(a,b)$$
and the Hausdorff distance between $A$ and $B$ is
$$h(A,B)=\max \{\overline{d}(A,B),\overline{d}(B,A)\}.$$
The Hausdorff distance is a metric on the collection $\operatorname{Comp}(X)$ of all compact nonempty subsets of $X$ and consequently on the collection $\cont(X)$ of all nonempty subcontinua of $X$. A continuum $X$ has the property of Kelley if for every $\varepsilon>0$ there exists $\delta>0$ for which whenever $x,y \in X$ are such that $d(x,y)<\delta$ and $A \in \cont(X)$ contains $x$, it is possible to find $B \in \cont(X)$ containing $y$ with $h(A,B)<\varepsilon$.

Another way to define the same topology as induced by the Hausdorff distance on $\operatorname{Comp}(X)$ is given by the Vietoris topology whose base is the collection
$$\{\lr{U_1,\ldots,U_n}\ | \ U_i \text{ is open in } X \text{ for each } i \text{ and } n \in \omega\}$$
where
$$\lr{U_1,\ldots,U_n}=\{F \in \operatorname{Comp}(X) \ | \ F \subseteq \cup_{i=1}^n U_i \text{ and } U_i \cap F \ne \emptyset \text{ for each } i\}.$$
We say that a property of subsets of a continuum $X$ is generic if the collection of subsets of $X$ satisfying this property is comeager in $\cont(X)$.

Fixed a continuum $X$, an order arc on $\cont(X)$ is a continuum $\mathcal{C} \subseteq \cont(X)$ such that for every pair $K_1,K_2 \in \mathcal{C}$ either $K_1 \subseteq K_2$ or $K_2 \subseteq K_1$. Note that this includes the unitary subsets of $\cont(X)$ and that if $\mathcal{C}$ is nondegenerate, then it is an arc. An order arc $\mathcal{C}$ is called maximal if $X \in \mathcal{C}$ and $\{x\}\in\mathcal{C}$ for some $x \in X$. In this case, $x$ is called the root of $\mathcal{C}$ and it is denoted by $\Root(\mathcal{C})$. It is known that a collection $\mathcal C\subseteq \cont(X)$ is maximal linearly ordered set (ordered by inclusion) if and only if $\mathcal C$ is a maximal order arc (see \cite[p. 111]{nadlerhyperspaces}). The set $\operatorname{OA}(X)=\{\mathcal{C} \subseteq \cont(X) \ | \ \mathcal{C} \text{ is an order arc}\}$ is a compact subspace of $\cont(\cont(X))$ and $\moa(X)=\{\mathcal{C} \in \operatorname{OA}(X) \ | \ \mathcal{C} \text{ is a maximal order arc}\}$ is a closed subset of it, thus a Polish space (see \cite{spacesoffforderarcs}). If $\mathcal{C}$ is an order arc between $K_1$ and $K_2$ in $\cont(X)$, it has a parametrization, one of which we will denote as $[K_1,K_2]:[0,1] \to \cont(X)$ when it is necessary to deal with different arcs at the same time. In order to distinguish the Hausdorff metric on $\cont(X)$ and $\cont(\cont(X))$ we will denote by $h^2$ the Hausdorff metric on $\cont(\cont(X))$ induced by $h$ on $\cont(X)$.

Given a continuum $X$ and $K_1,K_2 \in \cont(X)$, we say that $K_1$ is ambiently homeomorphic to $K_2$ if there exists a homeomorphism $\psi:X \to X$ such that $\psi(K_1)=K_2$. Two chains $\mathcal{C}_1,\mathcal{C}_2 \in \moa(X)$ are ambiently equivalent, if there is an automorphism $\psi$ of $X$ such that the induced homeomorphism $\Psi:\moa(X) \to \moa(X)$ given by
$$\Psi(\mathcal{C}) = \{\psi(K) \in \cont(W_M) \ | \ K \in \mathcal{C}\}$$
satisfies $\Psi(\mathcal{C}_1)=\mathcal{C}_2$.

\section{Generically hereditarily equivalent continua}
\subsection{Generalized Wa\.zewski dendrites}

 We start this section by proving a general property of dendrites.

 \begin{prop}\label{countableend}
  If $Y$ is a dendrite and $X \subseteq Y$ is a subdendrite, then $\ep(X) \setminus \ep(Y)$ is countable.
 \end{prop}
 \begin{proof}
  For each $x \in \ep(X) \setminus \ep(Y)$ we know that there exists at least one open connected component $C_x$ of $Y \setminus \{x\}$ that does not meet $X$. Also, if $x_1,x_2 \in \ep(X) \setminus \ep(Y)$ and $x_1 \ne x_2$, then $C_{x_1} \cap C_{x_2} =\emptyset$, so that
  $$\mathcal{X} = \{C_x \subseteq Y \ | \ x \in \ep(X) \setminus \ep(Y)\}$$
  is a collection of disjoint open sets of $Y$. Given that $Y$ is separable, $\mathcal{X}$ is countable and thus $\ep(X) \setminus \ep(Y)$ is countable.
 \end{proof}

 \begin{defi}
Let $X,Y$ be dendrites with $X \subseteq Y$. We say that a branching point $x$ of $X$ is maximal in $Y$ if $r_{Y,X}^{-1}(x)$ is degenerate, that is, if $r_{Y,X}^{-1}(x)=\{x\}.$
\end{defi}

\begin{lem}\label{maximalbranching}
Let $X,Y$ be dendrites with $X \subseteq Y$. A branching point $x$ of $X$ is maximal in $Y$ if and only if one of the following are satisfied:
\begin{enumerate}[label={(\roman*)}]
\item $r_{Y,X}^{-1}(x)$ is degenerate.
\item there is no arc $A \subseteq Y$ from $y \in Y\setminus X$ to $x$ with $A \cap X = \{x\}$.
\item every open neighborhood of $x$ in $X$ meets each component of $Y \setminus \{x\}$.
\end{enumerate}
\end{lem}
\begin{proof}
(i) $\Rightarrow$ (ii): If $r^{-1}_{Y,X}(x)$ is degenerate, for every $y \in Y \setminus X$ the arc $A=[y,x] \subseteq Y$ meets $X$ in a different point from $X$. Thus, $A \cap X \ne \{x\}$.

(ii) $\Rightarrow$ (iii): We will prove the contrapositive. If there is an open neighborhood $U$ of $x$ in $X$ that does not meet each component of $Y \setminus \{x\}$, we can consider it a connected neighborhood without loss of generality. Hence, there exists a component $C$ of $Y\setminus \{x\}$ for which $C \cap X \cap U = \emptyset$. We claim that $[y,x] \cap X=\{x\}$ for every $y \in C$. Suppose not. Then, there exists $x' \in X \setminus \{x\}$ with $x' \in [y,x]$, but this is a contradiction since the connected component of $x'$ and $y$ are the same in $Y \setminus \{x\}$. Thus $[y,x] \cap X= \{x\}$.

(iii) $\Rightarrow$ (i): If every open neighborhood of $x$ in $X$ meets each component of $Y \setminus \{x\}$, given $y \in Y\setminus \{x\}$ there exists $y'$ in $X \setminus \{x\}$ in the same component of $Y\setminus \{x\}$ than $y$. Thus, $[y',x) \subseteq X$ is in the same component of $y$ and so is $[y,y']$ and $[y,x)$. Note that $x \not\in [y,y']$ since they are in the same component of $Y \setminus \{x\}$. Therefore, $[y,x) \cap X = \emptyset$ implies that $[y,x] \cup [y',y] \cup [y',x]$ is a subcontinuum of $Y$ with a closed arc, a contradiction.
\end{proof}

\begin{defi}
 Let $Y$ be a dendrite. If $X \subseteq Y$ is a dendrite, we say that $X$ is full if for every $b \in \mathcal{B}(Y) \cap X$, it holds that $b \in \mathcal{B}(X)$ and $b$ is maximal in $Y$. We denote
$$\operatorname{Full}(Y)=\left\{K \in \cont(Y) \ \left| \ K \text{ is  full in } Y\right.\right\}.$$
\end{defi}

\begin{prop}{}\label{fullcomeager}
 If $X$ is a dendrite, then $\operatorname{Full}(X)$ is a $G_\delta$ dense subset of $\cont(X)$.
\end{prop}
\begin{proof}
 For every $b \in \mathcal{B}(X)$, enumerate the connected components of $X \setminus \{b\}$. Then, let
$$\scr{M}(b,\ell) = \{K \in \cont(X) \ | \ b \in K \text{ and } K \text{ does not meet the $\ell$-th component of } X \setminus\{b\}\}.$$
Each $\scr{M}(b,\ell)$ is closed since if $K \not\in \scr{M}(b,\ell)$, then either $K$ does not contain $b$ or $K$ contains $b$ and meets the $\ell$-th component. In the first case $\lr{X,\{b\}^c}$ is an open set disjoint from $\scr{M}(b,\ell)$ that contains $K$. In the second case, let the open set $U$ be the component of $X \setminus \{b\}$ containing the $\ell$-th component of $K \setminus \{b\}$ so that $K \in \lr{X,U}$ and $\lr{X,U} \cap \scr{M}(b,\ell)=\emptyset$. Also, we claim that the sets $\scr{M}(b,\ell)$ have an empty interior. Given $K \in \scr{M}(b,\ell)$ and $\varepsilon>0$, there exists a connected open neighborhood $V$ of $b$ contained in $B(b,\varepsilon/2)$. Hence, $K' = K \cup \overline{V} \in \cont(X) \setminus \scr{M}(b,\ell)$ is at distance smaller than $\varepsilon$ from $K$ and $\scr{M}(n,\ell)$ has an empty interior. Therefore,
\begin{align*}
\scr{G}'&= \cont(X) \setminus\left( \left( \bigcup_{b \in \scr{B}(X)} \bigcup_{\ell} \scr{M}(b,\ell) \right) \cup \operatorname{Fin}_1(X) \right)
\end{align*}
is $G_\delta$ dense, hence comeager in $\cont(X)$.

Finally, we claim that $\operatorname{Full}(X)=\scr{G}'$. If $K \in \scr{M}(b,\ell)$, then $K \not\in \operatorname{Full}(X)$, since $b$ is either a branching point that is not maximal in $X$, because $K$ does not meet the $\ell$-th component of $X \setminus K$, or it may not be even a branching point of $K$. Thus, $\operatorname{Full}(X) \subseteq \scr{G}'$. On the other hand, if $K \in \operatorname{Full}(X),$ it is nondegenerate and not in any $\scr{M}(b,\ell)$ since by definition $b \in \mathcal{B}(K)$ and $b$ is maximal in $X$ whenever $b \in K$ hence $K$ meets every component of $X setminus \{b\}$. This means $\scr{G}' \subseteq \operatorname{Full}(X)$ and $\operatorname{Full}(X)=\scr{G}'$ is comeager in $\cont(X)$.
\end{proof}

\begin{theo}\label{generalizeddendrite}
For every $K \in \full(W_M)$ it holds that $K \simeq W_M$ and $\ep(K) \cap \mathcal{B}(W_M)=\emptyset$.
\end{theo}
\begin{proof}
Every $K \in \operatorname{Full}(W_M)$ is a non-degenerate subdendrite of $W_M$, thus if an arc is contained in $K$, it contains densely many branching points of every order as a subset of $W_M$, and by defintion of $\operatorname{Full}(W_M)$ each of these branching points are branching points of $K$ of the same order, hence $K \simeq W_M$. Moreover, if $K \in \operatorname{Full}(W_M)$, no branching point of $W_M$ can be an endpoint of $K$ since it has to be maximal.
\end{proof}

\begin{cor}{}\label{WMGHEC}
    $W_M$ is GHEC.
\end{cor}
\begin{proof}
    We know from Proposition \ref{fullcomeager} that $\operatorname{Full}(W_M)$ is a comeager subset of $\cont(W_M)$. Hence $W_M$ is GHEC by Theorem \ref{generalizeddendrite}.
\end{proof}

 \begin{defi}{}{}
A continuum $X$ is strongly GHEC if
$$\{K \in \cont(X) \ | \ K \simeq X\}$$
contains a comeager subset $\scr{H}$ of $\cont(X)$ such that whenever $K_1,K_2 \in \scr{H}$ there exists a homeomorphism $\psi:X \to X$ such that $\psi(K_1)=K_2$.
\end{defi}

We will prove that $W_M$ is strongly GHEC. First, we need to introduce a new concept:

\begin{defi}
 Given three points $x,y,$ and $z$ of a dendrite $X$, we say that $z$ is \textbf{between} $x$ and $y$ if $z \in [x,y]$. If $X$ and $Y$ are dendrites and $f:X \to Y$ is a function, we say that $f$ preserves betweenness on $A \subseteq X$ if for every $\{x,y,z\} \subseteq A$ with $z$ between $x$ and $y$ it holds that $f(z)$ is between $f(x)$ and $f(y)$.
\end{defi}

\begin{prop}{\cite[p. 237]{duchesnekaleidoscopic}}\label{duchesnehomeo}
Let $X,Y$ be dendrites without a free arc. If $f: \mathcal{B}(X) \to \mathcal{B}(Y)$ is a bijection that preserves the betweenness relation, then there is a homeomorphism $h:X \to Y$ extending $g$.
\end{prop}

The proof of the following lemma is left to the reader.

\begin{lem}\label{nonemptyinterior}
Let $X$ and $Y$ be two dendrites with $X \subseteq Y$. If $x \in X$, then $r_{Y,X}^{-1}(x) \setminus \{x\}$ is open in $Y$.
\end{lem}

\begin{prop}\label{WMnowheredense}
  The collection of nowhere dense subcontinua of $W_M$ is a $G_\delta$ dense subset of $\cont(W_M)$.
\end{prop}
\begin{proof}

 Given that branching points of $W_M$ are dense, we can say that $K \in \cont(W_M)$ is nowhere dense if $B(b,1/n) \not\subseteq K$ for every $b \in \mathcal{B}(W_M)$ and $n \in \mathbb{N}$. Thus, we define
 $$\mathscr{R}(b,n) = \{K \in \cont(W_M)\ | \ B(b,n) \subseteq K\}.$$
 We claim that $\mathscr{R}(b,n)$ is closed and has an empty interior, so that
 $$\cont(W_M) \setminus \left(\bigcup_{b \in \mathcal{B}(W_M)} \bigcup_{n \in \mathbb{N}} \scr{R}(b,n) \right)$$
 will be the comeager set made of nowhere dense subcontinua of $W_M$.

 Fixed $b \in \mathcal{B}(W_M)$ and $n \in \mathbb{N}$, $\scr{R}(b,n)$ is closed since for any sequence $(K_m) \subseteq \scr{R}(b,n)$ that converges to some set, say $K'$, it holds that $\overline{B(b,n)} \subseteq K_m$ for every $m \in \mathbb{N}$ and
 $$h(K',K_m) = h(K' \cup \overline{B(b,1/n)},K_m),$$
 thus $\overline{B(b,1/n)} \subseteq K'$ and $K' \in \scr{R}(b,n)$. Now, given $K \in \scr{R}(b,n)$, we can find $y \in B(b,1/n) \cap \ep(W_M)$ and for $\varepsilon>0$ an open neighborhood $V \subseteq B(y,\varepsilon/2) \cap B(b,1/n)$ of $y$ with $|\partial V|=1$ so that $W_M \setminus V$ is connected. Thus, $K' = K \cap W_M \setminus V$ is a continuum and $h(K,K')<\varepsilon$ by definition, hence $\scr{R}(b,n)$ has an empty interior.
\end{proof}

\begin{prop}{}{}
 Comeager many subcontinua of $\cont(W_M)$ contain an endpoint of $W_M$.
\end{prop}
\begin{proof}

From Theorem \ref{generalizeddendrite} we have that the collection of homeomorphic copies of $W_M$ is a comeager subset of $\cont(W_M)$. Given that a copy $K$ of $W_M$ has uncountably many endpoints and that only countably many endpoints of $K$ may lay outside of $\ep(W_M)$ by Proposition \ref{countableend}, comeager many subcontinua of $W_M$ meets $\ep(W_M)$.
\end{proof}

\begin{theo}{}\label{stronglyWM}
 $W_M$ is strongly GHEC. Precisely, the collection of subcontinua $K$ of $W_M$ satisfying
 \begin{enumerate}[label={(\roman*)}]
     \item $K \in \operatorname{Full}(W_M)$.
     \item $K$ is nowhere dense in $W_M$.
 \end{enumerate}
 is $G_\delta$ dense and for any pair $K_1,K_2$ in this collection there exists a homeomorphism $\psi:W_M \to W_M$ with $\psi(K_1)=K_2$.
\end{theo}
\begin{proof}

The first part of the theorem follows from Corollary \ref{WMGHEC} and Proposition \ref{WMnowheredense}. Thus, we proceed to show that the desired homeomorphism exists.

 Let $K_1$ and $K_2$ be subcontinua of $W_M$ satisfying (i) and (ii). First, we will construct a bijection $\varphi:\mathcal{B}(W_M) \to \mathcal{B}(W_M)$ preserving betweenness, such that
 $$\varphi(\mathcal{B}(W_M) \cap K_1) = \mathcal{B}(W_M) \cap K_2.$$
 We fix an enumeration $\{b_n\}$ of $\mathcal{B}(W_M)$ and define inductively functions $\varphi_n : X_n \to \mathcal{B}(W_M)$ and $\psi_n: Y_n \to \mathcal{B}(W_M)$, where 
 \begin{itemize}
 \item $\{b_1,\ldots,b_n\} \cup \psi_{n-1}(Y_{n-1}) \subseteq X_n \subset \mathcal{B}(W_M)$ and $\{b_1,\ldots,b_n\} \cup \varphi_n(X_n) \subseteq Y_n \subset \mathcal{B}(W_M)$.
 \item $\varphi_{n-1} \subseteq \varphi_n$ and $\psi_{n-1} \subseteq \psi_n$.
 \item $\varphi_n \circ \psi_{n-1} = \id_{Y_{n-1}}$ and $\psi_n \circ \varphi_n=\id_{X_n}$.
 \item $Y_{n-1} \cup \{y\}$ defines a tree $T_{n-1,n-1}^2$ and $\psi_{n-1}(Y_{n-1})\cup \{x\}$ defines a tree $T_{n-1,n-1}^1$ on $W_M$. The domain of $\psi_{n-1}$ is such that if $x_1,x_2 \in Y_{n-1}$, then the endpoint $z$ of $[x_1,y] \cap [x_2,y]$ other than $y$ is also in $Y_{n-1}$ and $\psi_{n-1}(z)=w$ where $w$ is the endpoint of $[\psi_{n-1}(x_1),x] \cap [\psi_{n-1}(x_2),x]$ other than $x$.
 \item $X_{n} \cup \{x\}$ defines a tree $T_{n,n-1}^1$ and $\varphi_{n}(X_{n})\cup \{y\}$ defines a tree $T_{n,n-1}^2$ on $W_M$. The domain of $\varphi_{n}$ is such that if $x_1,x_2 \in X_{n}$, then the endpoint $z$ of $[x_1,x] \cap [x_2,x]$ other than $x$ is also in $X_{n}$ and $\varphi_{n}(z)=w$ where $w$ is the endpoint of $[\varphi_{n}(x_1),y] \cap [\varphi_{n}(x_2),y]$ other than $y$.
 \item Betweenness is preserved by $\varphi_{n}$ and $\psi_{n}$.
 \item $\varphi_n(b) \in K_2$ if and only if $b \in K_2$.
 \item $\psi_n(b) \in K_1$ if and only if $b \in K_2$.
 \end{itemize}
 Then, we will have
 $$\varphi=\bigcup_{n \in \bb{N}} \varphi_n \quad \text{and} \quad \psi=\bigcup_{n \in \bb{N}} \psi_n$$
 so that both are functions defined on $\mathcal{B}(W_M)$ and $\psi = \varphi^{-1}$.
 
 By Theorem \ref{generalizeddendrite} we can fix $x \in K_1 \cap \ep(W_M)$ and $y \in K_2 \cap \ep(W_M)$. First, let $X_1=b_1$. We define $\varphi_1(b_1)$ as any branching point of $W_M$ having the same order as $b_1$ under the condition that $b_1 \in K_1$ if and only if $\varphi_1(b_1) \in K_2$, which is possible since $K_1$ and $K_2$ are nowhere dense in $W_M$. Next, we define $\psi_1(b_1)$ (and possibly the image of some other point so that $Y_1$ consists of $b_1$ and this other point) as follows:
 \begin{enumerate}[label={(\roman*)}]
  \item If $b_1 \in [y,\varphi_1(b_1)] \cap K_2$, then note that $[x,b_1] \cap K_1$ is an arc $[x,z]$, since $x \in \ep(W_M) \cap K_1$. Hence, we can choose $\psi_1(b_1)$ as any branching point of $[x,z]$ having the same order than $b_1$.
  \item If $b_1 \in [y,\varphi_1(b_1)] \cap K_2^c$, then $\varphi_1(b_1) \not\in K_2$ and $b_1 \not\in K_1$. Hence, $[x,b_1] \cap K_1^c$ contains an arc $[z,b_1]$, and we choose $\psi_1(b_1)$ as any branching point $b \in (z,b_1)$ of the same order as $b_1$.
  \item If $b_1 \not\in [y,\varphi_1(b_1)]$ and $b_1 \in K_2$. Then
  $$z=r_{W_M,[y,\varphi_1(b_1)]}(b_1) \in K_2.$$
  Thus, we define $\psi_1(z)$ as before and note that
  $$r_{K_1,[x,\psi_1(z)]}^{-1}(\psi_1(z)) \setminus [x,b_1]$$
  is a nonempty open set of $K_1$ since $K_1 \simeq W_M$, thus there is a branching point
  $$b \in r^{-1}_{K_1,[x,\psi_1(z)]}(\psi_1(z)) \setminus [x,b_1]$$
  of the same order as $b_1$ and this is chosen as $\psi_1(b_1)$. In this case, $Y_1=\{b_1,z\}$.
  \item If $b_1 \not\in [y,\varphi_1(b_1)]$ and $b_1 \not\in K_2$, let
  $$z = r_{W_M,[y,\varphi_1(b_1)]}(b_1)$$
  and define $\psi_1(z)$ accordingly to its case as in (i) or (ii). Then,
  $$Z = r_{W_M,[x,\psi_1(z)]}^{-1}(\psi_1(z)) \setminus ([x,b_1] \cup K_1)$$
  is open and nonempty in $W_M$ (since $[x,b_1] \cup K_1$ is nowhere dense in $W_M$). Hence, we can find $z' \in Z \cap \ep(W_M)$ so that $[\psi_1(z),z'] \setminus K_1$ is nondegenerate, thus we can find a branching point $b$ in $[\psi_1(z),z'] \setminus K_1$ having the same order of $b_1$ and this point can be set as $\psi_1(b_1)$.
 \end{enumerate}

\begin{figure}[h]
\centering
 \begin{tikzpicture}[scale=1.2]
 \begin{scope}[xshift=-5.5cm,yshift=2.25cm]
  \draw (-0.25,-0.25) rectangle (0.25,0.25);
  \node at (0,0) {\small 1};
 \end{scope}
 \begin{scope}
 \draw[fill=white] (180:5) circle (1.5pt) node[above] {\small $x$};
  \foreach \i/\j in {180/5,155/4,130/3,105/2,70/1,15/0.65,-30/0.4,-85/0.2}{
  \draw (0,0)--(\i:\j);
  \foreach \k/\l in {30/1,60/0.5,100/0.3,135/0.2}{
  \draw (\i:\j*0.5)--++({\i-\k}:{\l*(\j/3)});
  \foreach \m/\n in {30/1,60/0.5,100/0.3,135/0.2}{
  \coordinate (A) at (\i:\j*0.5);
  \path (A) +({\i-\k}:{\l*(\j/3)/2}) coordinate (B);
  \draw (B)--++({\i-\k-\m}:{\n*\l*(\j/6)});
  }
  \foreach \r in {0.25,0.75}{
  \foreach \k/\l in {30/1,60/0.5,100/0.3,135/0.2}{
  \draw (\i:\j*\r)--++({\i-\k}:{\l*(\j/6)});
  }
  }
  }
  }
  \end{scope}
  \begin{scope}[scale=0.5]
  \draw[line width=1.5pt] (-10,0)--(-3,0);
  \foreach \i/\j in {150/1,120/0.5,80/0.3,45/0.2}{
  \draw[line width=1.25pt] (-7.5,0)--++(\i:{\j*1.25});
  \draw[line width=1.25pt] (-5,0)--++(\i:{\j*1.665-(150-\i)/1000});
  }
  \end{scope}
  \coordinate (b1a) at (180:3.75);
  \coordinate (b1b) at (150:0.75);
  \coordinate (b1) at ($(b1a)+(b1b)$);
  \draw[fill=white] (b1) circle (1.5pt) node[above] {\small $b_1$};
 \begin{scope}[xshift = 7cm]
 \begin{scope}[xshift=-5.5cm,yshift=2.25cm]
  \draw (-0.25,-0.25) rectangle (0.25,0.25);
  \node at (0,0) {\small 2};
 \end{scope}
 \begin{scope}
  \foreach \i/\j in {180/5,155/4,130/3,105/2,70/1,15/0.65,-30/0.4,-85/0.2}{
  \draw (0,0)--(\i:\j);
  \foreach \k/\l in {30/1,60/0.5,100/0.3,135/0.2}{
  \draw (\i:\j*0.5)--++({\i-\k}:{\l*(\j/3)});
  \foreach \m/\n in {30/1,60/0.5,100/0.3,135/0.2}{
  \coordinate (A) at (\i:\j*0.5);
  \path (A) +({\i-\k}:{\l*(\j/3)/2}) coordinate (B);
  \draw (B)--++({\i-\k-\m}:{\n*\l*(\j/6)});
  }
  \foreach \r in {0.25,0.75}{
  \foreach \k/\l in {30/1,60/0.5,100/0.3,135/0.2}{
  \draw (\i:\j*\r)--++({\i-\k}:{\l*(\j/6)});
  }
  }
  }
  }
  \end{scope}
  \begin{scope}[scale=0.8,rotate=-25]
  \draw[fill=white] (180:5) circle (1.5pt) node[left] {\small $y$};
  \coordinate (b1a) at (180:3.75);
  \coordinate (b1b) at (150:0.75);
  \coordinate (b1) at ($(b1a)+(b1b)$);
  \draw[fill=white] (b1) circle (1.5pt) node[above] {\small $\varphi_1(b_1)$};
  \end{scope}
  \begin{scope}[scale=0.4,rotate=-25]
  \draw[line width=1.5pt] (-10,0)--(-3,0);
  \foreach \i/\j in {150/1,120/0.5,80/0.3,45/0.2}{
  \draw[line width=1.25pt] (-7.5,0)--++(\i:{\j*1.25});
  \draw[line width=1.25pt] (-5,0)--++(\i:{\j*1.665-(150-\i)/1000});
  }
  \end{scope}
  \coordinate (b1a) at (180:3.75);
  \coordinate (b1b) at (150:0.75);
  \coordinate (b1) at ($(b1a)+(b1b)$);
  \draw[fill=white] (b1) circle (1.5pt) node[above] {\small $b_1$};
  \end{scope}
 \begin{scope}[yshift=-3.5cm]
 \begin{scope}[xshift=-5.5cm,yshift=2.25cm]
  \draw (-0.25,-0.25) rectangle (0.25,0.25);
  \node at (0,0) {\small 3};
 \end{scope}
 \begin{scope}
  \foreach \i/\j in {180/5,155/4,130/3,105/2,70/1,15/0.65,-30/0.4,-85/0.2}{
  \draw (0,0)--(\i:\j);
  \foreach \k/\l in {30/1,60/0.5,100/0.3,135/0.2}{
  \draw (\i:\j*0.5)--++({\i-\k}:{\l*(\j/3)});
  \foreach \m/\n in {30/1,60/0.5,100/0.3,135/0.2}{
  \coordinate (A) at (\i:\j*0.5);
  \path (A) +({\i-\k}:{\l*(\j/3)/2}) coordinate (B);
  \draw (B)--++({\i-\k-\m}:{\n*\l*(\j/6)});
  }
  \foreach \r in {0.25,0.75}{
  \foreach \k/\l in {30/1,60/0.5,100/0.3,135/0.2}{
  \draw (\i:\j*\r)--++({\i-\k}:{\l*(\j/6)});
  }
  }
  }
  }
  \end{scope}
  \begin{scope}[scale=0.8,rotate=-25]
  \draw[fill=white] (180:5) circle (1.5pt) node[left] {\small $y$};
  \coordinate (b1a) at (180:3.75);
  \coordinate (b1b) at (150:0.75);
  \coordinate (b1) at ($(b1a)+(b1b)$);
  \draw[fill=white] (b1) circle (1.5pt) node[above] {\small $\varphi_1(b_1)$};
  \draw[line width=1.5pt] (b1)--(180:3.75)--(180:5);
  \draw[line width=1.5pt] (180:5)--(0,0);
  \draw[fill=white] (180:3.75) circle (1.5pt) node[below left] {\small $z$};
  \end{scope}
  \coordinate (b1a) at (180:3.75);
  \coordinate (b1b) at (150:0.75);
  \coordinate (b1) at ($(b1a)+(b1b)$);
  \draw[fill=white] (b1) circle (1.5pt) node[above] {\small $b_1$};
  \draw[line width=1.5pt] (0,0)--(180:3.75)--++(150:0.75);
 \begin{scope}[xshift=7cm]
\begin{scope}[xshift=-5.5cm,yshift=2.25cm]
  \draw (-0.25,-0.25) rectangle (0.25,0.25);
  \node at (0,0) {\small 4};
 \end{scope}
 \begin{scope}
  \foreach \i/\j in {180/5,155/4,130/3,105/2,70/1,15/0.65,-30/0.4,-85/0.2}{
  \draw (0,0)--(\i:\j);
  \foreach \k/\l in {30/1,60/0.5,100/0.3,135/0.2}{
  \draw (\i:\j*0.5)--++({\i-\k}:{\l*(\j/3)});
  \foreach \m/\n in {30/1,60/0.5,100/0.3,135/0.2}{
  \coordinate (A) at (\i:\j*0.5);
  \path (A) +({\i-\k}:{\l*(\j/3)/2}) coordinate (B);
  \draw (B)--++({\i-\k-\m}:{\n*\l*(\j/6)});
  }
  \foreach \r in {0.25,0.75}{
  \foreach \k/\l in {30/1,60/0.5,100/0.3,135/0.2}{
  \draw (\i:\j*\r)--++({\i-\k}:{\l*(\j/6)});
  }
  }
  }
  }
  \end{scope}
  \draw[fill=white] (180:5) circle (1.5pt) node[above] {\small $x$};
  \coordinate (b1a) at (180:3.75);
  \coordinate (b1b) at (150:0.75);
  \coordinate (b1) at ($(b1a)+(b1b)$);
  \draw[fill=white] (b1) circle (1.5pt) node[above left,xshift=0.2cm] {\small $\psi_1(\varphi_1(b_1))$};
  \draw[line width=1.5pt] (b1)--(180:3.75)--(180:5);
  \begin{scope}[rotate=-50,scale=0.6]
  \coordinate (b1a) at (180:3.75);
  \coordinate (b1b) at (150:0.5);
  \coordinate (b1) at ($(b1a)+(b1b)$);
  \draw[fill=white] (b1) circle (2.5pt) node[above,yshift=0.2cm] {\small $\psi_1(b_1)$};
  \draw[line width=1.5pt] (230:6.25)--(0,0)--(180:3.75)--++(150:0.5);
  \end{scope}
  \draw[fill=white] (180:3.75) circle (1.5pt) node[below left] {\small $\psi_1(z)$};
 \end{scope}
 \end{scope}
 \end{tikzpicture}
 \caption{Scheme of the process after the induction's base case: In 1 we have $K_1$ represented by the thicker curve, while $K_2$ is represented in 2. In 3 and 4 we have, respectively, the trees $T_{1,1}^2$ and $T_{1,1}^1$.}
\end{figure}

Now we consider that $\varphi_{n-1}$ and $\psi_{n-1}$ are defined, and we will define $\varphi_n$. If $b_n \in T_{n-1,n-1}^1$, we define $X_n=\{b_1,\ldots,b_n\} \cup \psi_{n-1}(Y_{n-1})$, define $\varphi_n|_{X_{n-1}}=\varphi_{n-1}$ and $\varphi_n|_{\psi_{n-1}(Y_{n-1})}$ such that $\varphi_n \circ \psi_{n-1} = \operatorname{id}_{Y_{n-1}}$. To define $\varphi_n(b_n)$ we split in cases:

\begin{enumerate}[label={(\roman*)}]
  \item $b_n \in X_{n-1}$: Nothing has to be done.

 \item $b_n \in \psi_{n-1}(Y_{n-1}) \setminus X_{n-1}:$ We have that $\varphi_n(b_n)$ is $b_m \in Y_{n-1}$ for which $\psi_{n-1}(b_m)=b_n$. Hence, betweenness is preserved since $\psi_{n-1}$ preserves betweenness.

 \item $b_n \not\in \psi_{n-1}(Y_{n-1})$:
 
  This means $b_n$ is not a node of $T_{n-1,n-1}^1$, hence it lies in the minimal interval $(b_{n_1},b_{n_2})$ where $b_{n_1},b_{n_2} \in \psi_{n-1}(Y_{n-1})$. If $b_n \in K_1$ then either $b_{n_1}$ or $b_{n_2}$ is in $K_1$. Assuming without loss of generality that $b_{n_1} \in K_1$, we know from induction hypothesis that $\varphi_{n-1}(b_{n_1}) \in K_2$ and $\varphi_{n-1}(b_{n_1}) \not\in \ep(K_2)$, hence $K_2 \cap [\varphi_{n-1}(b_{n_1}),\varphi_{n-1}(b_{n_2})]$ is a nondegenerate interval and it contains a branching point $b$ of the same order of $b_n$, thus $\varphi_n(b_n)=b$. If $b_n \not\in K_1$, either $b_{n_1} \not\in K_1$ or $b_{n_2} \not\in K_2$, otherwise we could find a closed curve in $W_M$. Assuming without loss of generality that $b_{n_2} \not\in K_1$, we know from induction hypothesis that $\varphi_{n-1}(b_{n_2}) \not \in K_2$, thus $[\varphi_{n-1}(b_{n_1}),\varphi_{n-1}(b_{n_2})] \cap K_2$ is nondegenerate and we can find a branching point $b$ of the same order as $b_{n}$, hence define $\varphi_n(b_n)=b$.
\end{enumerate}

We claim that the construction preserves betweenness among the points in $X_{n-1}$ and $b_n$. Given $b_{k_1}$ and $b_{k_2}$ in $X_{n-1}$, we can assume either $b_n \in [b_{k_1},b_{k_2}]$ or $b_{k_1} \in [b_n,b_{k_2}]$. In the first case we have $b_n \in [b_{n_1},b_{n_2}]$ and by definition $b_{n_1},b_{n_2} \in [b_{k_1},b_{k_2}]$, so $\varphi_n(b_{n_1}),\varphi_n(b_{n_2}) \in [\varphi_n(b_{k_1}),\varphi_n(b_{k_2})]$ by hypothesis and betweenness is preserved. In the second case we have $b_{k_1} \in [b_n,b_{k_2}]$ so that
 $$b_{n} \in [b_{n_1},b_{n_2}] \subseteq [b_{n_1},b_{k_1}] \subseteq [b_{n_1},b_{k_2}].$$
 Since all these points $b_{n_1},b_{n_2},b_{k_1},b_{k_2}$ are already in $X_{n-1}$, their betweenness relation is preserved, so $\varphi_n(b_{k_1}) \in [\varphi_n(b_{n_2}),\varphi_n(b_{k_2})]$ and hence $\varphi_n(b_{k_1}) \in [\varphi_n(b_{n}),\varphi_n(b_{k_2})]$.

If $b_n \not\in T_{n-1,n-1}^1$, let
$$z = r_{W_M,T_{n-1,n-1}^1}(b_n)$$
and define $\varphi_n(z)$ as above so that $X_n=X_{n-1} \cup \{z,b_n\}$. If $b_n \in K_1$, then $z \in K_1$ and thus
$$r_{K_2,Y}^{-1}(\varphi_n(z))$$
contains an open subset of $K_2$, where
$$Y=T_{n-1,n-1}^2 \cap K_2.$$
Therefore, we can find there a branching point $b$ having the same order as $b_n$, and this is $\varphi_n(b_n)$. If $b_n \not\in K_1$,
$$r_{W_M,T_{n-1,n-1}^2}^{-1}(\varphi_n(z)) \setminus K_2$$
contains an open set since $K_2$ is closed and nowhere dense in $W_M$. Thus, we can find there a branching point $b$ of the same order of $b_n$ and define it as $\varphi_n(b_n)$.

For betweenness, again consider $b_{k_1},b_{k_2}$ in $X_{n-1}$, but in this case only $b_{k_1} \in [b_n,b_{k_2}]$ is possible, since $[b_{k_1},b_{k_2}] \subseteq T_{n-1,n-1}^1$ and $b_n \not\in T_{n-1,n}^1$ so $b_n \not\in [b_{k_1},b_{k_2}]$. We know that
 $$z = r_{W_M,T_{n-1,n-1}^1}(b_n) \in T_{n-1,n-1}^1$$
 was added to the domain and by construction $b_{k_1} \in [z,b_{k_2}]$, so $\varphi_n(b_{k_1}) \in [\varphi_n(z),\varphi_n(b_{k_2})]$ from previous case and
 $$\varphi_n(b_{k_1}) \in [\varphi_n(b_n),\varphi_n(b_{k_2})].$$

 Finally, we need to preserve the induction hypothesis, but the construction makes $[x,b_n] \cap T_{n-1,n-1}^1 = [x,z]$ and thus either the endpoint other than $x$ in
 $$[b_m,x]\cap [z,x]$$
 is already a node of $T^1_{n-1,n-1}$ or it is $z$, which is a node now. The same holds for $T^2_{n-1,n-1}$.

 An analogous argument is used to define $Y_n$ and $\psi_{n}(Y_n)$ from $T^2_{n,n-1}$, yielding trees $T^1_{n,n}$ and $T^2_{n,n}$. Then, we fall again on the inductive case, which means we can define the function $\varphi:\mathcal{B}(W_M) \to \mathcal{B}(W_M)$ that preserves betweenness relation and also the belonging relation to $K_1$ and $K_2$.
 
 From Proposition \ref{duchesnehomeo}, $\varphi$ can be extended to a homeomorphism $\psi: W_M \to W_M$. Given that $\varphi$ is a bijection satisfying
 $$\varphi(\mathcal{B}(W_M) \cap K_1) \subseteq K_2 \quad \text{and} \quad \varphi^{-1}(\mathcal{B}(W_M) \cap K_2) \subseteq K_1,$$
 we have
 $$K_2 \cap \mathcal{B}(W_M) = \varphi(\varphi^{-1}(K_2)) \subseteq \varphi(K_1 \cap \mathcal{B}(W_M)) \subseteq K_2 \cap \mathcal{B}(W_M).$$
 Therefore
 $$K_2 = \overline{\psi(K_1 \cap \mathcal{B}(W_M))} = \psi(\overline{K_1 \cap \mathcal{B}(W_M)}) = \psi(K_1).$$

\end{proof}

Additionally, it can be proven that the collection of nowhere dense full copies of $W_M$ is a zero-dimensional (totally disconnected) subspace of $\cont(W_M)$.

\subsection{Peano GHEC}

Now that we know that there are GHEC spaces that are not HEC. We can ask if there is a characterization of GHEC. For that, we will use the characterization.

\begin{prop}[Characterization of dendrites - {\cite[p. 189 - Exercise 10.50]{nadlercontinuum}}]{}\label{dendritept}

A Peano continuum is tree-like if and only if it is a dendrite.
\end{prop}

On the other hand, there is another characterization of tree-like spaces. Recall that, given a finite open cover $\scr{U}=\{U_1,\ldots,U_n\}$ of $X$, the nerve of $\scr{U}$ is the graph having as nodes $v_1,\ldots,v_n$ and such that there exists an edge between $v_i$ and $v_j$ if and only if $U_i \cap U_j \ne \emptyset$.

\begin{prop}
    A continuum $X$ is tree-like if and only if for all $\varepsilon>0$ there exists a finite open cover $\scr{U}$ of $X$ whose mesh is smaller than $\varepsilon$ and has a tree as a nerve.
\end{prop}

Additionally, we can prove the following:

\begin{lem}{}\label{treeisgdelta}
Let $X$ be a continuum. Then tree-like continua form a $G_\delta$ subset of $\cont(X)$.
\end{lem}
\begin{proof}
Let $n \in \bb{N}$ and $\scr{C}_n$ be the class of all finite collections of open subsets of $X$ having mesh smaller than $1/n$ and a tree as a nerve. Define
$$\mathcal{U}_n = \bigcup \{ \lr{U_1,\ldots,U_k} \subset \cont(X)\ | \ \{U_1,\ldots,U_k\} \in \scr{C}_n\}$$
Each $\mathcal{U}_n$ is open and the set
$$U = \bigcap_{n \in \bb{N}} \mathcal{U}_n$$
is the collection of all tree-like subcontinua of $X$, thus it is a $G_\delta$ set.
\end{proof}

\begin{theo}\label{peanoimpliestreegeneric}
Let $X$ be a Peano continuum. Then tree-like subcontinua form a comeager set in $\cont(X)$.
\end{theo}
\begin{proof}
    By Theorem \ref{treeisgdelta} we know that the collection of tree-like subcontinua is $G_\delta$ in $\cont(X)$. Thus, we only need to show that trees are dense in $X$. 

    Given $K \in \cont(X)$, for every $\varepsilon > 0$ there exists a finite collection $\{U_1,\ldots,U_n\}$ of connected open subsets of $X$ of diameter smaller than $\varepsilon$ that cover $K$. Assume without loss of generality that $U_i \cap K \ne \emptyset$ for every $i$, so $K \in \lr{U_1,\ldots,U_n}$ and $U =U_1 \cup \cdots \cup U_n$ is connected. For each $i$, take $a_i \in U_i \cap K$. Given that $\{a_1,\ldots,a_n\}$ is connected, we can find an arc $r_j \subseteq U$ from $a_1$ to $a_j$ for each $j \in \{2,\ldots, n\}$. Hence, $R=r_2 \cup \cdots \cup r_n$ is a Peano continuum and by construction $h(R,K)<\varepsilon$. Now we will recursively define a tree using $R$. Start with $r_2$. If $a_3 \in r_2$, nothing has to be done in the first step. If $a_3 \not\in r_2$, let $s_3$ be an arc on $R$ from $a_3$ to a point $y \in r_2$ such that $s_3 \cap r_2 = \{y\}$. $T_3=r_2 \cup s_3$ is a tree. Suppose that the tree $T_m$ is defined until the point $a_m$ where $m<n$. If $a_{m+1} \in T_m$, $T_{m+1}=T_m$. If not, we define $T_{m+1}=T_m \cup s_{m+1}$ where $s_{m+1}$ is any arc on $R$ from $a_{m+1}$ to $y \in T_m$ such that $s_{m+1} \cap T_m = \{y\}$. By definition, $T_{m+1}$ is a tree and $T_n$ is a tree such that $h(T,K)<\varepsilon$ since $T \subseteq U$ and $\{a_1,\ldots,a_n\} \subseteq T$. Therefore, trees are dense in $\cont(X)$.
\end{proof}

\begin{cor}\label{PeanoGHEC}
If $X$ is a nondegenerate Peano GHEC, then $X$ is an arc or a dendrite such that for every $m \in \bb{N} \cup \{\omega\}$ the collection of branching points of order $m$ is either dense or empty.
\end{cor}
\begin{proof}
By Lemma \ref{peanoimpliestreegeneric}, there is a comeager subset of $\cont(X)$ made of tree-like continua. Given that $X$ is GHEC, we can then find a comeager subset of homeomorphic copies of $X$ that are tree-like, so $X$ is a tree-like Peano continuum, hence a dendrite by Proposition \ref{dendritept}. In what follows, suppose that $X$ is a nondegenerate dendrite which is not an arc. This implies that $X$ cannot contain a free arc, since this would imply the existence of an open subset $\cont(X)$ containing only arcs and points, which are not homeomorphic to $X$. Therefore, branching points are arcwise dense. 

From Proposition \ref{fullcomeager}, we have that comeager many subcontinua of $X$ are full in $X$. With this, we claim that if $X$ has a branching point of order $m$, then the collection of branching points of order $m$ is dense in $X$. Suppose by contradiction that $U$ is an open subset of $X$ that does not contain any branching point of order $m$. This open set induces a nonempty open set $\lr{U}$ in $\cont(X)$, but no element of $\scr{H} \cap \lr{U}$ can be a copy of $X$ since there is no point of order $m$ in $U$ and every branching point of a dendrite in $\scr{H}$ has the same order as in $X$. Also, the collection of homeomorphic copies of $X$ cannot be disjoint to $\scr{H}$, otherwise $\cont(X)$ would be meager, which is a contradiction since it is a compact metric space. Therefore, if $X$ has a branching point of order $m$, branching points of order $m$ are dense.
\end{proof}

\begin{cor}
If $X$ is a GHEC dendrite whose branching points are of a single order $m$, then $X \simeq W_m$.
\end{cor}
\begin{proof}
    By Corollary \ref{PeanoGHEC} the branching points of $X$ are arcwise dense and therefore $X \simeq W_m$ since all the branching points have the same order.
\end{proof}

Note that we still do not have that the collection of branching points of every order is arcwise dense in $X$, therefore we do not know whether every Peano GHEC that is not an arc is a generalized Wa\.zewski dendrite.

\subsection{GHEC and further generic properties}

In \cite{basso2024surfaces} it is shown that every homogeneous Peano continuum other than the circle does not admit a generic chain. From this idea, we can ask if being GHEC is a sufficient condition to have a generic chain. This has led us to the definition of the following concepts:

\begin{defi}{}{}\label{defgenericchain}
Let $X$ be a continuum, we say that
\begin{itemize}
\item GCHEC holds for $X$ if and only if
$$\{\mathcal{C} \in \moa(X) \ | \ \forall C \in \mathcal{C}, C \text{ nondegenerate implies } C \simeq X\}$$
is a comeager subset of $\moa(X)$.
\item GCGHEC holds for $X$ if and only if
$$\{\mathcal{C} \in \moa(X) \ | \ \forall^\ast C \in \mathcal{C}, C \text{ nondegenerate implies } C \simeq X \}$$
is a comeager subset of $\moa(X)$.
\end{itemize}
\end{defi}

\begin{prop}{}{}
If $X$ is HEC, then $X$ is GHEC and GCHEC holds for $X$.
\end{prop}
\begin{proof}
If $X$ is HEC, then 
$$\{F \in \cont(X) \ | \ F \simeq X\} = \cont(X) \setminus \operatorname{Fin}_1(X)$$
and since $\operatorname{Fin}_1(X)$ is closed in $\cont(X)$, we just have to show that $\operatorname{Fin}_1(X)$ has an empty interior and $X$ will be GHEC. Given any point $x \in X$, the Boundary Bumping Theorem gives us nondegenerate continua of arbitrarily small diameter containing $x$. Hence, every open set in $\cont(X)$ containing $\{x\}$ has a nondegenerate element, and thus is not contained in $\operatorname{Fin}_1(X)$, that is, $\operatorname{Fin}_1(X)$ has an empty interior.

Moreover, if $X$ is HEC, then every element of any maximal order arc in $\cont(X)$ is either a point or a nondegenerate subcontinuum, but since every nondegenerate subcontinuum is homeomorphic to $X$, GCHEC holds for $X$.
\end{proof}

Note that the previous proposition does not give any relation between GHEC and the generic chain properties in any direction. The essence of what we need is to understand if the existence of a comeager set in $\cont(X)$ induces a comeager set in $\moa(X)$ or vice-versa. Therefore, the following theorem is useful:

\begin{theo}[Desintegration theorem in the category context - {\cite[p. 163]{generictsankov}}]{}\label{comeagerfunction}

If $Z$ and $Y$ are Polish spaces and $f:Z \to Y$ is continuous and comeager (i.e. the preimage of every comeager set is a comeager set), then a set $S$ with Baire property is comeager in $Z$ if and only if $S \cap f^{-1}(y)$ is comeager in $f^{-1}(y)$ for comeager many $y$ in $Y$.
\end{theo}

We will apply the Desintegration theorem to the set
$$\mathcal{Z}=\{(\mathcal{C},K) \ | \ \mathcal{C} \in \moa(X), K \in \mathcal{C}\} \subseteq \moa(X) \times \cont(X).$$
$\moa(X)$ is a Polish space, hence, $\moa(X) \times \cont(X)$ is a Polish space. Also, it can be easily observed that $\mathcal{Z}$ is closed. Hence, $\mathcal{Z}$ is a Polish space. Now, note that $\pi_1:\mathcal{Z} \to \moa(X)$ given by the restriction of the projection from the product into the first coordinate is continuous and comeager. To see that it is comeager, note that we only have to show that the preimage of a dense set is dense. If $D \subseteq \moa(X)$ is dense, then
$$\pi_1^{-1}(D) = \{(\mathcal{C},K) \in \mathcal{Z} \ | \ \mathcal{C} \in D\},$$
so given any $(\mathcal{C}',K') \in \mathcal{Z}$ and $\varepsilon>0$, there is $\mathcal{C} \in D$ for which $h_2(\mathcal{C}',\mathcal{C})<\varepsilon$. Hence, some $K \in \mathcal{C}'$ satisfies $h(K,K')<\varepsilon$ and
$$(\mathcal{C},K) \in \pi_1^{-1}(D) \cap B_{h_2}(\mathcal{C}',\varepsilon) \times B_h(K')$$
so $\pi_1^{-1}(D)$ is dense. Therefore, from the Theorem \ref{comeagerfunction}, we have that for every comeager set $S$ in $\mathcal{Z}$ we can associate a comeager set $G$ in $\moa(X)$.

Now, to apply the previous discussion, the idea is to understand when $\pi_2:\mathcal{Z} \to \cont(X)$ is comeager. If $\pi_2^{-1}(G)$ is comeager for a given comeager set $G \subseteq \cont(X)$, we know from the Disintegration Theorem that for comeager many $\mathcal{C} \in \moa(X)$
$$\pi_1^{-1}(\mathcal{C}) \cap \pi_2^{-1}(G)$$
is comeager in
$$\pi_1^{-1}(\mathcal{C}) = \{(\mathcal{C},K) \ | \ K \in \mathcal{C}\} \simeq \mathcal{C}.$$
Therefore, if $X$ is GHEC, we have that GCGHEC holds for $X$ whenever $\pi_2$ is comeager. We will use this to prove an implication when $X$ is hereditarily indecomposable. Recall the following Theorem:

\begin{theo}{\cite[p. 148]{nadlerhyperspaces}}{}
Let $X$ be a continuum. Then $X$ is hereditarily indecomposable if and only if $\cont(X)$ is uniquely arcwise connected.
\end{theo}

\begin{lem}{}\label{himoa}
If $X$ is hereditarily indecomposable continua, then every maximal order arc in $\cont(X)$ is of the form
$$\mathcal{A}_x=\{K \in \cont(X) \ | \ x \in K\}$$
that is, $\moa(X) = \{\mathcal{A}_x \ | \ x \in X\}.$
\end{lem}
\begin{proof}
We know that $\mathcal{A}_x$ is a maximal linearly ordered subset of $\cont(X)$ for every $x \in X$. Indeed, if for $K_1,K_2 \in \mathcal{A}_x$ we have neither $K_1 \subseteq K_2$ nor $K_2 \subseteq K_1$, it means that $K_1 \cup K_2$ is a decomposable subcontinuum of $X$, a contradiction. Hence $\mathcal{A}_x$ is a maximal order arc. If $\mathcal{A}$ is a maximal order arc containing $\{x\}$, then every element of $\mathcal{A}$ contains $x$, thus $\mathcal{A} \subseteq \mathcal{A}_x$. We claim $\mathcal{A}=\mathcal{A}_x$. If not, then there exists $K_0 \in \mathcal{A}_x \setminus \mathcal{A}$, so $\mathcal{A}$ is not a maximal linearly ordered set, a contradiction. Thus, $\mathcal{A}=\mathcal{A}_x$.
\end{proof}

\begin{lem}{}\label{fiscomeager}
If $X$ is hereditarily indecomposable, the preimage of a dense set by $\pi_2:\mathcal{Z} \to \cont(X)$ is also a dense set.
\end{lem}
\begin{proof}
Given that $X$ is hereditarily indecomposable, it has Kelley's property. Let $O$ be a dense set in $Y$. If $(\mathcal{A}_x,K) \in \mathcal{Z}$ and $B_{h_2}(\mathcal{A}_x,\varepsilon_1) \times B_h(K,\varepsilon_2)$ is an open neighborhood $U$ of $(\mathcal{A}_x,K)$, there are two possibilities: $K \in O$ or $K \not\in O$. If $K \in O$, then $(\mathcal{A}_x,K) \in f^{-1}(O)$ and we are done. If $K \not\in O$, then let $\varepsilon=\min\{\varepsilon_1,\varepsilon_2\}$ and $\delta$ given by Kelley's property for such $\varepsilon$. We know that
$$B_h(K,\delta) \cap O \ne \emptyset,$$
and for $K'$ in the intersection, we have two cases:
\begin{itemize}
    \item $K' \cap K \ne \emptyset$: Since $X$ is hereditarily indecomposable, either $K \subseteq K'$ or $K' \subseteq K$. Hence, by Lemma \ref{himoa}, $K' \in C$ and
    $$(\mathcal{A}_x,K') \in \pi_2^{-1}(O) \cap U.$$
    \item $K' \cap K = \emptyset$: Since $h(K',K)<\delta$, if $\{x\}$ is the starting point of $\mathcal{A}_x$, then there exists $y \in K'$ such that $d(x,y)<\delta$. Therefore, for each $F \in \mathcal{A}_x$ we can find $K_F \in \cont(X)$ containing $y$ such that $h(F,K_F)<\varepsilon$. By Proposition \ref{himoa} the collection $\{K_F\}_{F \in \mathcal{A}_x}$ is contained in the unique maximal order arc starting from $\{y\}$, hence we have that
    $$\sup_{F \in \mathcal{A}_x} \inf_{G \in \mathcal{A}_y} h(F,G)<\varepsilon.$$
    Applying Kelley's property and Lemma \ref{himoa}
    $$\sup_{G \in \mathcal{A}_y} \inf_{F \in \mathcal{A}_x} h(G,F)<\varepsilon$$
    also holds, and $h_2(\mathcal{A}_x,\mathcal{A}_y)<\varepsilon.$ Therefore
    $$(\mathcal{A}_y,K') \in \pi_2^{-1}(O) \cap U.$$
\end{itemize}
With this, we have proven that $U$ contains some point of $f^{-1}(O)$, and thus $f^{-1}(O)$ is dense.
\end{proof}

Note that we have just showed that $\pi_2: \mathcal{Z} \to \cont(X)$ is comeager when $X$ is hereditarily indecomposable. By using the preceding results and related discussion, we have the following proposition:

\begin{prop}{}{}
If $X$ is GHEC and hereditarily indecomposable, then GCGHEC holds for $X$.
\end{prop}

We can also work similarly applying the Disintegration Theorem on the other direction, but some attention is required because at first we do not know much about the comeager subset of $\cont(X)$ given by it.

\begin{theo}{}{}
    If $X$ is hereditarily indecomposable and GCHEC holds for it, then $X$ is GHEC.
\end{theo}
\begin{proof}
    Since GCHEC holds for $X$, we know that
    $$\scr{C}=\{(\mathcal{C},K) \in \mathcal{Z} \ | \ \forall K' \in \mathcal{C} \text{ nondegenerate }K' \simeq X\}$$
    is comeager in $\mathcal{Z}$ since $\pi_1$ is a comeager function. Given that $X$ is hereditarily indecomposable, $\pi_2$ is also comeager by Lemma \ref{fiscomeager}, and we can apply the Disintegration Theorem (\ref{comeagerfunction}) to find out that for comeager many $K \in \cont(X)$, $\pi^{-1}_2(K) \cap \scr{C}$ is comeager in $\pi^{-1}_2(K)$. Thus, for any such nondegenerate $K$, $\pi^{-1}_2(K)$ is nonempty, therefore $K \simeq X$ and $X$ is GHEC.
\end{proof}

Therefore, the following diagram holds:

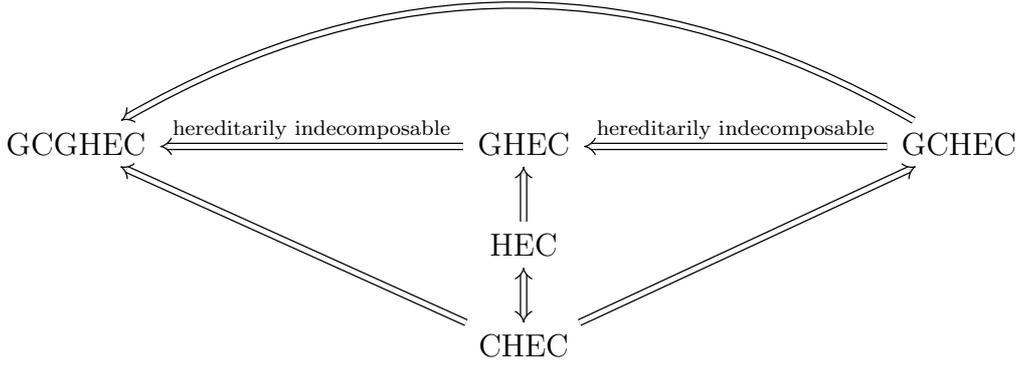
\begin{center}
\begin{tikzcd}[column sep = 4cm]
\text{GCGHEC} & \text{GHEC} \arrow[l,Rightarrow,swap,"\text{hereditarily indecomposable}"]& \text{GCHEC} \arrow[ll,Rightarrow,bend right] \arrow[l,Rightarrow,swap,"\text{hereditarily indecomposable}"]\\
& \text{HEC} \arrow[u,Rightarrow] \arrow[lu,Rightarrow] \arrow[ru,Rightarrow] &
\end{tikzcd}
\end{center}

\section{Generic maximal chains of $W_M$}

\subsection{Properties of generic chains of $MOA(W_M)$}
We know that $W_M$ is GHEC, so it is natural to ask whether GCGHEC holds for $W_M$. We will show that GCGHEC holds, and explore some properties to which the arcs can be restricted and still have a comeager collection of maximal order arcs.

\begin{prop}\label{WMGCGHEC}
GCGHEC holds for $W_M$. Moreover, the collection of maximal order arcs $\scr{C}$ such that
$$\mathcal{C} \in \scr{C} \iff \forall^\ast K \in \mathcal{C}(K\in \full(W_M))$$
is comeager in $\moa(W_M)$.
\end{prop}
\begin{proof}
It suffices to prove the second statement. We will use the Disintegration Theorem in the category context (Theorem \ref{comeagerfunction}). Let
$$\mathcal{Z} = \{(\mathcal{C},K) \in \moa(W_M) \times \cont(W_M) \ | \ K \in \mathcal{C}\}.$$
We start proving that the collection
$$\mathcal{Z}' = \{(\mathcal{C},K) \in \mathcal{Z} \ | \ K \in \full(W_M)\}$$
is comeager in $\mathcal{Z}$. The projection $\pi_2 : \mathcal{Z} \to \cont(W_M)$ is continuous, thus, given that
\begin{align*}
    \full(W_M) &= \left\{K \in \cont(W_M) \ \left| \ \begin{array}{c}K \simeq W_M \text{ and if } b \in \mathcal{B}(W_M) \cap K,\\ \text{ then } b \in \mathcal{B}(K) \text{ and it is maximal in } W_M\end{array}\right.\right\}\\
    & = \cont(W_M) \setminus\left( \left( \bigcup_{b \in \scr{B}(W_M)} \bigcup_{\ell} \scr{M}(b,\ell) \right) \cup \operatorname{Fin}_1(W_M) \right),
\end{align*}
and that the sets $\mathcal{M}(b,\ell)$ and $\operatorname{Fin}_1(W_M)$ are nowhere dense and closed, we will prove that the preimage of these sets by the projection $\pi_2 : \mathcal{Z} \to \cont(X)$ has an empty interior, so from continuity of $\pi_2$
$$\mathcal{Z}' = \pi_2^{-1}(\full(W_M))$$
will be comeager in $\mathcal{Z}$.

In the case of $\operatorname{Fin}_1(W_M)$, its preimage is the collection of pairs $(\mathcal{C},\{x\})$ where $x \in X$ and $\mathcal{C}$ is a maximal order arc starting at $\{x\}$. Clearly this collection has an empty interior since $\operatorname{Fin}_1(W_M)$ is has an empty interior in $\cont(W_M)$. Now, given $b \in \mathcal{B}(W_M)$ and the $\ell$-th connected component of $W_M \setminus \{b\}$, let $(\mathcal{C},K) \in \pi_2^{-1}(\scr{M}(b,\ell))$. The goal is to find for every $\varepsilon>0$ a chain $\mathcal{C}'$ and $K' \in \cont(W_M)$ such that \begin{enumerate}
    \item $(\mathcal{C}',K') \in \pi_2^{-1}(\scr{M}(b,\ell))^c$;
    \item $h^2(\mathcal{C},\mathcal{C}')<\varepsilon$;
    \item $h(K,K')<\varepsilon$;
    \item $K' \not\in \scr{M}(b,\ell)$.
\end{enumerate}
Given that $b$ is a limit point of the $\ell$-th component of $W_M \setminus \{b\}$ and $W_M$ is a Peano continuum, for any $\varepsilon>0$ there exists $y$ in the $\ell$-th component such that $d(b,[y,b])<\varepsilon$. Hence, $K'=K \cup [b,y] \not\in \scr{M}(b,\ell)$ and $h(K,K')<\varepsilon$ so we define an order arc $[K,K']:[0,1] \to \cont(W_M)$ given by
$$[K,K'](t) = K \cup (\cup_{s \leq t}[b,y](s)).$$
We also define an order arc $[K',W_M]$ as the image of
$$[K',W_M](t) = K' \cup [\{x\},W_M](t) \text{ for } t \in [0,1].$$
Then,
$$\mathcal{C}' = [\{x\},K]\cup [K,K'] \cup [K',W_M]$$
is such that $h^2(\mathcal{C},\mathcal{C}')<\varepsilon$ since for every $Y \in \mathcal C'$ there exists $X \in \mathcal{C}$ such that $Y \subseteq X$ and $Y \setminus X \subseteq [b,y]$. By construction $(\mathcal{C}',K') \in \pi_2^{-1}(\scr{M}(b,\ell))^c$ and we are done. This means that $\mathcal{Z}'$ is comeager.

Since $\mathcal{Z}'$ is a comeager subset of the Polish space $\mathcal{Z}$, the Disintegration Theorem implies the existence of a comeager family $\mathcal{Y}$ of maximal order arcs such that whenever $\mathcal{C} \in \mathcal{Y}$
$$\pi_1^{-1}(\mathcal{C}) \cap \mathcal{Z}' = \{(\mathcal{C},K) \ | \ K \in \scr{C}'\}$$
is comeager in $\pi_1^{-1}(\mathcal{C})$, that is,
$$\{K \in \mathcal{C} \ | \ K \in \mathcal{C}' \}$$
is comeager in $\mathcal{C}$. Therefore, the claim of the theorem holds and GCGHEC holds for $W_M$.
\end{proof}

\begin{lem}\label{arcsinclosedendrites}
Let $X$ be a dendrite, $K$ be a nondegenerate subdendrite of $X$, $[x,y]$ and arc in $K$ from $x$ to $y$, and $a_1 \ne a_2 \in (x,y)$. If
$$0<\varepsilon< \{d(a_1,a_2),d(x,a_1),d(x,a_2), d(y,a_1), d(y,a_2), \operatorname{diam}(K)/3 \},$$
then there exists $0<\delta<\varepsilon$ such that for every $K' \in \cont(X)$ with $h(K,K')<\delta$ there exists an arc $[a_1',a_2'] \subseteq [a_1,a_2] \cap K'$.
\end{lem}
\begin{proof}
Let $U_i \subseteq B(a_i,\varepsilon/2)$ be a connected open neighborhood of $a_i$ in $X$ for $i=1,2$ and fix $0<\delta<\varepsilon/2$ such that $B(a_i,\delta) \subseteq U_i$. Given $K' \in B_h(K,\delta)$, there exist $z_1,z_2 \in K'$ with $d(a_1,z_1),d(a_2,z_2) < \delta$. Now, $X$ is arcwise connected and each $U_i$ is connected, so there are arcs $[z_i,a_i] \subseteq U_i$. Moreover,
$$[z_1,a_1] \cap [z_2,a_2] = \emptyset$$
for $i \ne j$, since $U_1 \cap U_2 = \emptyset$. Hence,
$$[z_1,a_1] \cup [a_1,a_2] \cup [a_2,z_2] \subseteq K'$$
contains the arc $[z_1,z_2] \subseteq X$ which is unique because $X$ does not contain closed curves. Every subcontinuum of a dendrite is a dendrite, thus $[z_1,z_2] \subseteq K'$ and we claim that $[z_1,z_2] \cap [a_1,a_2]$ contains an arc.

Indeed, given that $[z_1,a_1] \cup [z_2,a_2]$ is disconnected, $[z_1,z_2] \cap [a_1,a_2]$ is nonempty, and it is nondegenerate, otherwise, if $[z_1,z_2] \cap [a_1,a_2] = \{w\}$,
$$d(a_1,a_2)\leq d(a_1,w)+d(w,a_2) = \varepsilon/2 + \varepsilon/2 = \varepsilon,$$
which is a contradiction. Dendrites are hereditarily unicoherent, hence $[z_1,z_2] \cap [a_1,a_2]$ is a non-degenerate subcontinuum of $[z_1,z_2]$, hence an arc.
\end{proof}

\begin{lem}\label{limitofWMisWM}
If $\mathcal{C} \in \moa(W_M)$ and the collection of $K \in \mathcal{C}$ with $K \simeq W_M$ is dense in $\mathcal{C}$, then every nondegenerate element of $\mathcal{C}$ is homeomorphic to $W_M$.
\end{lem}
\begin{proof}

Every $K \in \mathcal{C}$ is a dendrite, so we will prove that if $K$ is nondegenerate, then the branching points of all orders $m \in M$ are arcwise dense in $K$ by proving that every arc contains a branching point of order $m$. Fix an arc $[x,y] \subseteq K$ and let $a_1 \ne a_2 \in (x,y)$. For some compatible metric on $W_M$ fix
$$\varepsilon< \{d(a_1,a_2),d(x,a_1),d(x,a_2), d(y,a_1), d(y,a_2), \operatorname{diam}(K)/3,h(K, \Root(\mathcal{C}))\}.$$
From Lemma \ref{arcsinclosedendrites}, there exists $0<\delta<\varepsilon$ for which for any $K' \in \cont(X)$ with $h(K,K')<\delta$ and $K' \subseteq K$ there exists an arc $[a_1',a_2'] \subseteq [a_1,a_2] \cap K'$. Hence, since the collection of $K \in \mathcal{C}$ such that $K \simeq W_M$ is dense in $\mathcal{C}$, we can find $K' \subsetneq K$ homeomorphic to $W_M$ satisfying
$$h(K',K) < \delta.$$
The characterization of $W_M$ implies that $K' \cap [a_1,a_2]$ contains an arc with densely many branching points of order $m$ so $[x,y]$ contains a branching point of order $m$.
\end{proof}

\begin{theo}
GCHEC holds for $W_M$.
\end{theo}
\begin{proof}
From Proposition \ref{WMGCGHEC} we know that GCGHEC holds for $W_M$ and from previous lemma if the set of homeomorphic copies of $W_M$ is dense in a chain in $\cont(W_M)$, then all of its nondegenerate elements are homeomorphic to $W_M$. Therefore, GCHEC holds for $W_M$.
\end{proof}

From now on we will introduce concepts and properties that will not only give a comeager collection of chains that grow in a more strict way, but also that will be essential to show that the chains ambiently homeomorphic.

\begin{defi}
Given $\mathcal{C} \in \moa(W_M)$ and $x \in W_M$, the hitting time of $x$ in $\mathcal{C}$ is $K \in \mathcal{C}$ such that $x \in K$ and $x \not\in K'$ for all $K' \in \mathcal{C}$ with $K' \subsetneq K$.
\end{defi}

\begin{lem}
Given $\mathcal{C} \in \moa(W_M)$ and $x \in W_M$, then $x$ is an endpoint of its hitting time in $\mathcal{C}$.
\end{lem}
\begin{proof}
We will prove by contradiction. If $x$ is not an endpoint of its hitting time $K$ in $\mathcal{C}$, then it is a cut point, and also a cut point of $W_M$. This means that $W_M \setminus \{x\}$ has at least two components and, for every $K' \in \mathcal{C}$ with $K' \subsetneq K$, $K'$ is cointained in the same component $C$ of $W_M \setminus \{x\}$ as the starting point of the chain. However, there is a point $y$ of $K$ in another component $C'$ of $W_M$, thus we can find $\varepsilon>0$ for which $B(y,\varepsilon) \subseteq C'$. Therefore,
$$B(y,\varepsilon) \cap K' = \emptyset$$
but this means $h(K,K')\geq \varepsilon$ for all $K' \subsetneq K$, which is a contradiction since $\mathcal{C}$ is a chain.
\end{proof}

\begin{defi}
Given $\mathcal{C} \in \moa(W_M)$ and $x \in X$, let $K$ denote the hitting time of $x$ in $\mathcal{C}$. The hitting level of $x$ in $\mathcal{C}$ is
$$T(x,\mathcal{C})=\{y \in W_M \ | \ K \text{ is the hitting time of } y\}.$$
\end{defi}

\begin{defi}
Let $\mathcal{C}_1,\mathcal{C}_2 \in \moa(W_M)$. Given $A \subseteq W_M$ a function $f:A \to A$ preserves hitting levels if
$$f(T(x,\mathcal{C}_1) \cap A) \subseteq T(f(x),\mathcal{C}_2) \quad \text{ and } \quad f^{-1}(T(f(x)),\mathcal{C}_2) \subseteq T(x,\mathcal{C}_1)$$
for every $x \in X$.
\end{defi}

\begin{defi}
A chain $\mathcal{C} \in \moa(W_M)$ is willful if for every arc $A \subseteq W_M$ and $K_1,K_2 \in \mathcal{C}$ with
\begin{enumerate}
    \item $K_1 \subsetneq K_2$,
    \item $\emptyset \ne K_1 \cap A \subsetneq A$,
\end{enumerate}
then $K_1 \cap A \subsetneq K_2 \cap A$.
\end{defi}

\begin{lem}
The following is equivalent for any given $\mathcal{C} \in \moa(W_M)$:
\begin{enumerate}[label={\roman*)}]
\item $\mathcal{C}$ is willful.
\item Let $\{x\}$ be the starting point of $\mathcal{C}$. For every arc $[x,b] \subseteq W_M$, where $b \in \mathcal{B}(W_M)$, and every $K_1,K_2 \in \mathcal{C}$ with
\begin{enumerate}
    \item $K_1 \subsetneq K_2$,
    \item $\emptyset \ne K_1 \cap A \subsetneq A$,
\end{enumerate}
it holds that $K_1 \cap A \subsetneq K_2 \cap A$.
\end{enumerate}
\end{lem}
\begin{proof}
We have to prove only that ii) implies i). Given any arc $A = [z_1,z_2] \subseteq W_M$ and $K_1,K_2 \in \mathcal{C}$ with
\begin{enumerate}[label={(\alph*)}]
    \item $K_1 \subsetneq K_2$,
    \item $\emptyset \ne K_1 \cap A \subsetneq A$,
\end{enumerate}
we have that either $z_1$ or $z_2$ are not in $K_1$, otherwise $K_1$ would contain $A$. Suppose without loss of generality that $z_2 \not\in K_1$, so that
$$Y=[x,z_2] \setminus K_1$$
is nondegenerate, hence contains densely many branching points. Let $b \in Y \cap \mathcal{B}(W_M)$. The arc $[x,b]$ satisfies the condition of item ii), hence $K_1 \cap [x,b] \subsetneq K_2 \cap [x,b]$. Let
$$y=r_{W_M,[z_1,z_2]}(x)$$
so that $[x,y] \subseteq K_1$ since $x \in K_1$ and $K_1 \cap [z_1,z_2] \ne \emptyset$. This means that $K_1 \cap [y,z_2] \subsetneq K_2 \cap [y,z_2]$, and thus
$$K_1 \cap A \subsetneq K_2 \cap A.$$
\end{proof}

\begin{prop}
Let $\mathcal{C}$ be a willful chain in $\moa(W_M)$. Then
\begin{enumerate}[label={(\roman*)}]
\item If $K_1,K_2 \in \mathcal{C}$ with $K_1 \subsetneq K_2$, then every point
$$b \in (K_1 \setminus \ep(K_1)) \cap \mathcal{B}(W_M)$$
is maximal in $K_2$.
\item If $K \in \mathcal{C}$ is nondegenerate, then $K \simeq W_M$.
\end{enumerate}
\end{prop}
\begin{proof}
\begin{enumerate}[label={\roman*)}]
    \item If $b \in Y$, then the hitting time $K_b$ of $b$ in $\mathcal{C}$ is such that $K_b \subsetneq K_1$. Thus, for every component of $W_M \setminus \{b\}$ that does not contain $K_b$ we can take a point $z$ and if $x=\Root(\mathcal{C})$, we have by definition of willfulness that $[x,z] \cap K_b \subsetneq [x,z] \cap K_1$. Therefore, $b \in \mathcal{B}(K_1)$ and it is maximal in $W_M$, hence in $K_2$.
    \item We know from (i) that every branching point of $\mathcal{B}(W_M)$ in $K \setminus \ep(K)$ is in $\mathcal{B}(K_1)$ and it is maximal in $W_M$, thus it has the same order in $K$ as in $W_M$. Therefore, given an arc $A \subseteq K$, $A \cap \mathcal{B}(K) = A \cap \mathcal{B}(W_M)$, so $K$ has arcwise densely many branching points of every order in $M$ and it is homeomorphic to $W_M$.
\end{enumerate}
\end{proof}

\begin{prop}
    Let $\mathcal{C} \in \moa(W_M)$ be a chain where for every $K_1 \ne K_2 \in \mathcal{C}$ with $K_1 \subsetneq K_2$, $K_1$ is nowhere dense in $K_2$. Then for every $K \in \mathcal{C} \setminus \Root(\mathcal{C})$, 
    $$\operatorname{End}(K) \setminus (\cup \{K' \in \mathcal{C} \ | \ K' \subsetneq K\})$$
    is dense in $K$.
\end{prop}
\begin{proof}
Given that $\mathcal{C}$ is a chain and $K \ne \Root(\mathcal{C})$, there exists a sequence $(K_n) \subseteq \mathcal{C}$ for which $K_n \subseteq K_{n+1} \subseteq K$ and $\lim K_n = K$. Given that $K_n$ is nowhere dense in $K$, that $K$ is a continuum, and that $\ep(K)$ is a comeager subset of $K$, we have
$$\operatorname{End}(K) \setminus (\cup \{K' \in \mathcal{C} \ | \ K' \subsetneq K\})$$
dense in $K$. 
\end{proof}

\subsection{Generic chains are homeomorphically equivalent}

\begin{lem}\label{rootiscomeagermap}
Let $X$ be a Peano continuum and $D\subseteq X$.
Then
\begin{enumerate}[label={(\roman*)}]
    \item If $D$ is dense, then $\{\mathcal C: \Root(\mathcal C)\in D\}$ is dense.
    \item If $D$ is open, then $\{\mathcal C: \Root(\mathcal C)\in D\}$ is open.
\end{enumerate}
\end{lem}
\begin{proof}
\begin{enumerate}[label={(\roman*)}]
\item If $D$ is dense, then given $\mathcal{C} \in \moa(X)$ and $\varepsilon>0$ there exists a connected open neighborhood $U$ of $x=\Root(\mathcal{C})$ having diameter smaller than $\varepsilon$. Thus, we can find $y \in U \cap D$ and an arc $\gamma:[0,1] \to U$ from $y$ to $x$. Then, 
$$\mathcal{C}'=\{\gamma([0,t]) \ | \ 0 \leq t \leq 1\} \cup \{K \cup \gamma([0,1]) \ | \ K \in \mathcal{C}\}$$
is a maximal order arc with $\Root(\mathcal{C}')=y$ and $h^2(\mathcal{C},\mathcal{C}')<\varepsilon$. Therefore $\{\mathcal C: \Root(\mathcal C)\in D\}$ is dense.
\item It is enough to show that the map $\Root: \moa(X)\to X$ is continuous. If $(\mathcal{C}_n)$ is a sequence of chains converging to $\mathcal{C}$, then for every $\varepsilon>0$ there exists $N \in \bb{N}$ for which we can find $K_n \in \mathcal{C}_n$ satisfying $h(\Root(\mathcal{C}),K_n)<\varepsilon$ whenever $n \geq N$. In particular, $d(\Root(\mathcal{C}),\Root(\mathcal{C}_n))<\varepsilon$, so $(\Root(\mathcal{C}_n))$ converges to $\Root(\mathcal{C})$ and the map is continuous.
\end{enumerate}
\end{proof}

\begin{theo}\label{ambientlyeqWM}
The set of chains $\mathcal{C}\in \moa(W_M)$ satisfying
\begin{enumerate}[label={(\roman*)}]
\item The root of $\mathcal C$ is an endpoint of $W_M$.
\item If $K,L \in \mathcal{C}$ with $K \subsetneq L$, then $K$ is nowhere dense in $L$.
\item If $K\in\mathcal C$, then $|\ep(K)\cap\mathcal{B}(W_M)|\leq 1$.
\item $\mathcal{C}$ is willful.
\end{enumerate}
form a comeager set. Moreover, any two chains 
satisfying these conditions are ambiently equivalent.
\end{theo}

\begin{proof}
First, we will prove that each of the conditions (i)-(iv) yield comeager sets.

\underline{Proof of (i):}

By Lemma \ref{rootiscomeagermap} the collection of chains having root in a comeager set is comeager. Thus, given that $\ep(W_M)$ is comeager, the set of chains having an endpoint as root is comeager.

\underline{Proof of (ii):}

Let $\scr{B}$ be a countable base for the topology made of connected open sets. Define for $B,C,D,E \in \scr{B}$, where $B \cap D=\emptyset$, $\overline{C} \subset B$ and $\overline{E} \subset D$, the set
$$\scr{Z}_{B,C,D,E} = \left\{\mathcal{C} \in \moa(W_M) \ \left| \ \begin{array}{c} \exists K,L \in \mathcal{C} : K \subsetneq L, \emptyset \ne B \cap L \subseteq K,\\ K \cap \overline{C} \ne \emptyset, K \cap \overline{D}=\emptyset, L \cap \overline{E} \ne \emptyset\end{array}\right.\right\}.$$
Note that if $\mathcal{C} \in \scr{Z}_{B,C,D,E}$, then the sets $K,L \in \mathcal{C}$ that witness it are such that $K$ is not nowhere dense in $L$. On the other hand, if $\mathcal{C}$ is a chain such that for some $K,L \in \mathcal{C}$ with $K \subsetneq L$ it holds that $K$ has a non-empty interior in $L$, there exists an open set $B \in \scr{B}$ for which $\emptyset \ne B \cap K = B \cap L$, hence an open set $C \in \scr{B}$ with $\overline{C} \subset B$ and $K \cap \overline{C} \ne \emptyset$. Also, given that $K \subsetneq L$, there exists an open set $D \in \scr{B}$ for which $L \cap D \ne \emptyset$ and $K \cap D = \emptyset$, hence in particular there is $E \in \scr{B}$ with $\overline{E} \subset D$ and $L \cap \overline{E} \ne \emptyset$. Thus $\mathcal{C} \in \scr{Z}_{B,C,D,E}$. Therefore, if
$$\scr{Y} = \{\mathcal{C} \in \moa(W_M) \ | \ \text{ if } A,B \in \mathcal{C} \text{ and } A \subsetneq B, \text{ then } A \text{ is nowhere dense in } B\},$$
we have that
$$\scr{Y} = \moa(W_M) \setminus \bigcup_{B \in \scr{B}} \bigcup_{\substack{C \in \scr{B} \\ \overline{C} \subset B}} \bigcup_{D \in \scr{B}} \bigcup_{\substack{E \in \scr{B} \\ \overline{E} \subset D}} \scr{Z}_{B,C,D,E}.$$
We claim that each $\scr{Z}_{B,C,D,E}$ is closed and has an empty interior.

First, we will prove that $\scr{Z}_{B,C,D,E}$ is closed. Let $(\mathcal{C}_n)$ be a sequence in $\scr{Z}_{B,C,D,E}$ converging to $\mathcal{C}$. For each $\mathcal{C}_n$ there exists a pair of witnesses $K_n,L_n \in \mathcal{C}_n$ satisfying the conditions of $\scr{Z}_{B,C,D,E}$, and we can assume that the sequences $(K_n)$ and $(L_n)$ converge respectively to sets $K,L \in \mathcal{C}$. By definition, the sets $K$ and $L$ satisfy $K \subseteq L$, $K \cap D = \emptyset$ and $L \cap \overline{E} \ne \emptyset$, thus $K \subsetneq L$. Also, $K \cap \overline{C} \ne \emptyset$, so $L \cap B \ne \emptyset$. Hence, it remains to show $L \cap B \subseteq K$. Let $x \in L \cap B$ and let $(x_n)$ be a sequence converging to $x$ with $x_n \in L_n$. There exists $N \in \bb{N}$ for which $x_n \in B$ for every $n \geq N$ since $B$ is open, but then $x_n \in K_n$, and thus $x \in K$. Therefore, $L \cap B \subseteq K$ and $\scr{Z}_{B,C,D,E}$ is closed.

To show that $\scr{Z}_{B,C,D,E}$ has an empty interior, let $\mathcal{C} \in \scr{Z}_{B,C,D,E}$ and $\Gamma:[0,1] \to \mathcal{C}$ be a parametrization of $\mathcal{C}$ and $\varepsilon>0$. Let $r=\inf \Gamma^{-1}(\lr{W_M,\ol{C}})$. If $r>0$, let $y \in \Gamma(r) \cap \partial C$ and find an endpoint $z \in C$ such that $h([y,z],y)<\varepsilon$. Define $U \subset C$ with $\ol{U} \subseteq C$ as a connected open neighborhood of $z$ of diameter smaller than $\varepsilon$ such that $|\partial U|=1$. Denote $p=\partial U$. Then, define two auxiliary order arcs $\alpha:[r,1] \to \cont([y,p])$ going from $\{y\}$ to $[y,p]$ and $\beta:[1,2-r] \to \cont(\ol{U})$ going from $\{p\}$ to $\ol{U}$. Then, we define $\Gamma':[0,2-r] \to \cont(W_M)$ by
$$\Gamma'(t)=\begin{cases}
    \Gamma(t), \text{ if } t < r,\\
    \Gamma(r) \cup \alpha(t), \text{ if } r \leq t \leq 1,\\
    (\Gamma(t-1+r) \setminus U) \cup \beta(t), \text{ if } 1 \geq t \geq 2-r.
\end{cases}$$
Then $\Gamma'(t_1) \cap B \ne \Gamma'(t_2)$ whenever $\Gamma'(t_1) \cap \ol{C} \ne \emptyset$ and $t_1 < t_2$. Hence, the chain $\mathcal{C}'$ given by $\Gamma'$ is not in $\scr{Z}_{B,C,D,E}$ and by construction $h(\mathcal{C},\mathcal{C}')<\varepsilon$. If $r=0$, take $y=\Root(\mathcal{C})$ and $U$ with the additional condition that $y \not\in U$ so that the same argument works. Therefore, $\scr{Z}_{B,C,D,E}$  has an empty interior.

\underline{Proof of (iii):}

Let $n \ne m$ and define
 $$\mathscr{A}_{n,m}=\{\mathcal{C} \in \moa(W_M) \ | \ \exists K \in \mathcal{C} : \{b_n,b_m\} \subseteq \ep(K)\}.$$
 The set $\mathscr{A}_{n,m}$ is closed, since if $(\mathcal{C}_n)$ is a sequence in $\mathscr{A}_{n,m}$ converging to $\mathcal{C} \in \moa(W_M)$, we can find $K_i \in \mathcal{C}_i$ such that $\{b_n,b_m\} \in \ep(K_i)$ for every $i \in I$. By compactness there is a convergent subsequence and this subsequence converges to $K \in \mathcal{C}$. Then, we claim $\{b_n,b_m\} \in \ep(K)$. Both points belong to $K$, so suppose that one of them, $b_n$ for example, is not an endpoint of $K$. We know that $b_n$ is a cut point of $W_M$ and $K$ is non-degenerate. Hence, we claim that $K_i \setminus \{b_n\}$ is contained in one connected component of $W_M \setminus \{b_n\}$ for every $i \in I$. Every $K_i$ contains $(b_n,b_m)$, which is contained in one of the components of $W_M \setminus \{b_n\}$. Thus, since $K_i \setminus \{b_n\}$ is also connected, it has to be in the same component as $(b_n,b_m)$. However, since $b_n$ is not an endpoint of $K$, it is a cut point of it, thus there exists a point $x$ of $K$ in a component that does not meet $(b_n,b_m)$, and thus an open neighborhood of $x$ that is contained in this component, therefore not meeting any $K_i$, a contradicition. This means $\mathcal{C} \in \mathscr{A}_{n,m}$ and $\mathscr{A}_{n,m}$ is closed.

 Also, we claim that $\mathscr{A}_{n,m}$ has an empty interior. Given $\varepsilon>0$ and $\mathcal{C} \in \mathscr{A}_{n,m}$, let $x$ be the root of $\mathcal{C}$ and let $y$ be a point in a component of $W_M \setminus \{b_n\}$ that does not contain $x$ such that
 $$h([b_n,y],\{b_n\})< \varepsilon.$$
 Let $K_y$ be the hitting time of $y$ in $\mathcal{C}$ and choose a point $z \in (x,b_n)$ with
 $$h([z,b_n],b_n)<\varepsilon.$$
 Then, let $K_z$ be the hitting time of $z$ in $\mathcal{C}$. If $[z,y](t)$ is the arc from $\{z\}$ to $[z,y]$ growing within $[z,y]$, we can define a new maximal order arc $\mathcal{C}'$ as follows: Let $\Gamma: [0,1] \to \cont(W_M)$ be a parametrization of $\mathcal{C}$ with $\Gamma(t_z)=K_z$ and $\Gamma(t_y)=K_y$, and define $\Gamma':[0,2] \to \cont(W_M)$ as
 $$\Gamma'(t)=\begin{cases}
    \Gamma(t), \text{ if } 0 \leq t \leq t_z,\\
    \Gamma(t_z) \cup [z,y](t-t_z), \text{ if } t_z \leq t \leq 1+t_z,\\
    \Gamma(t-1) \cup [z,y], \text{ if } 1+t_z \leq t \leq 2.
   \end{cases}
   $$
Note that if $t \geq t_z$, then $z \in \Gamma'(t)$ and
$$h(\Gamma(t_z),\Gamma(t_z) \cup [z,y])< 2\varepsilon$$
since
$$h(\{z\},[z,y])\leq h(\{z\},[z,b_n])+h([z,b_n],[z,y]) = h(z,[z,b_n])+h(\{b_n\},[z,y]) < 2\varepsilon,$$
so $\Gamma'([0,2])$ is a maximal order arc $\mathcal{C}'$ for which
$$h^2(\mathcal{C},\mathcal{C}')<2 \varepsilon.$$
Moreover, $\mathcal{C}' \not\in \mathscr{A}_{n,m}$ since $b_m \not\in [x,y]$ and thus $b_m$ is not in $\Gamma'(1+t_z)$ while $b_n$ is. Therefore, $\mathscr{A}_{n,m}$ has an empty interior and
$$\moa(W_M) \setminus  \bigcup_{\substack{n,m \in \mathbb{N} \\ n \ne m}} \mathscr{A}_{n,m} = \{\mathcal{C} \in \moa(X) \ | \ |T(b,\mathcal{C}) \cap \mathcal{B}(W_M)|=1 \text{ for all } b \in \mathcal{B}(W_M)\}$$
is a comeager set.

\underline{Proof of (iv):}

We claim that fixed $n \in \bb{N}$ and $b \in \mathcal{B}(W_M)$,
 $$\scr{R}(n,b)=\left\{\mathcal{C} \in \moa(W_M) \ \left| \ \begin{array}{c}\exists K_1 K_2 \in \mathcal{C}, h(K_1,K_2)\geq 1/n , K_1 \subsetneq K_2, B(b,1/n) \cap K_1= \emptyset\\
 \text{ and } {[\Root(\mathcal{C}),b]} \cap K_1 = [\Root(\mathcal{C}),b] \cap K_2  
 \end{array}
\right.\right\}$$
is closed and has an empty interior. If we prove this, the set
$$\scr{G}=\moa(W_M) \setminus \left(\bigcup_{b \in \mathcal{B}(W_M)} \bigcup_{n \in \bb{N}} \scr{R}(n,b)\right)$$
is comeager and $\mathscr{G}$ is exactly the desired collection. Indeed, note that if $\mathcal{C} \in \mathscr{G}$, given $K_1, K_2 \in \mathcal{C}$ with $K_1 \subseteq K_2$ and any $b \in \mathcal{B}(W_M)$ with $b \not\in K_1$, there is $n \in \bb{N}$ for which
$$h(K_1,K_2) \geq 1/n \quad \text{and} \quad B(b,1/n) \cap K_1 = \emptyset.$$
We know that $\mathcal{C} \not\in \scr{R}(b,n)$, so
$$[\Root(\mathcal{C}),b] \cap K_1 \subsetneq [\Root(\mathcal{C}),b] \cap K_2.$$

\underline{\textit{Part 1:}} $\scr{R}(n,b)$ is closed. Let $(\mathcal{C}_i)$ be a sequence in $\scr{R}(n,b)$ converging to $\mathcal{C} \in \moa(W_M)$. Then, there are two sequences $({K}_i)$ and $({L}_i)$ witnessing the conditions of $\scr{R}(n,b)$ that have convergent subsequences to $K'$ and $L'$ respectively. Without loss of generality we assume that the subsequences are the same sequence as the original. We claim that $K'$ and $L'$ witness the conditions required to have $\mathcal{C} \in \scr{R}(n,b)$. Note that $K'\subsetneq L'$, $h(K',L') \geq 1/n$ and $B(b,1/n) \cap K'=\emptyset$.

First, we claim that $K',L' \in \mathcal{C}$: $\mathcal{C}$ is the limit of $(\mathcal{C}_i)$, so given $\varepsilon>0$, there exists $I \in \mathbb{N}$ such that $$h^2(\mathcal{C}_i,\mathcal{C})< \varepsilon$$
whenever $i \geq I$. However, by definition of $K'$ and $L'$, there exists $I'$ for which
$$\max\{h(K',{K}_i),h(L',{L}_i)\}<\varepsilon$$
whenever $i \geq I'$, hence if $i \geq \max \{I,I'\}$
$$h^2(\mathcal{C}_i,\mathcal{C} \cup \{K', L'\})<\varepsilon.$$
Thus, by uniqueness of the limit, $\mathcal{C}=\mathcal{C} \cup K' \cup L'$.

Let $x = \Root(\mathcal{C})$. It remains to show that $[x,b] \cap K'=[x,b]\cap L'$ so that $\mathcal{C} \in \scr{R}(n,b)$. By definition, $A_1 = [x,b] \cap K' \subseteq [x,b] \cap L' = A_2$. If $z \in A_2 \setminus A_1$, and consequently $z \in L' \setminus K'$, then there exists a sequence $(z_i)$ with $z_i \in {L}_i$ converging to $z$. Let $r=h(z,K')$, $U$ be a connected open neighborhood of $z$ of diameter smaller than $r/2$, and fix $i$ such that $$h({K}_i,K')<r/2 \quad \text{and} \quad z_i \in U.$$
Note that the endpoint $y_i$ other than $x$ in
$${K}_i \cap [x,b] = {L}_i \cap [x,b]$$
is not in $U$, since this would mean $h(z,K')<r$. However, note that
$$[x,z_i] \subseteq [x,z] \cup [z,z_i],$$
and since $[z,z_i] \subseteq U$ while $[x,z_i] \subseteq {L}_i$, there is a point
$$y \in [x,z_i] \cap [x,z] \cap [z,z_i] \subseteq U$$
which is in
$${L}_i \cap [x,z] \subseteq {L}_i \cap [x,b] = {K}_i \cap [x,b]=[x,y_i],$$
which is a contradiction because it implies $h(z,{K}_i)<r/2$ and thus $h(z,K')<r$.

\underline{\textit{Part 2:}} Now we want to prove that $\scr{R}(n,b)$ has an empty interior. Given $\varepsilon>0$ and $\mathcal{C} \in \scr{R}(n,b)$, let $C \in \mathcal{C}$ be the hitting time of $b$ and $x=\Root(\mathcal{C})$. Take a cover of $[\{x\},C] \subseteq \mathcal{C}$ given by
$$\{B_h(D,1/(5n)) \ | \ D \in [\{x\},C]\}$$
which yields a finite collection $\{D_1,\ldots,D_{k(n)}\}$ with $D_1=\{x\}$ and $D_{k(n)}=C$ such that
$$D_i \subseteq D_{i+1} \quad \text{and} \quad h(D_i,D_{i+1})< 1/(2n).$$
Then, cover $[x,b]$ by
$$\{U_y \subseteq B_d(y,\varepsilon/3) \ | \ y \in [x,b], U_y \text{ open connected neighborhood of } y \text{ in } [x,b]\}$$
from which we can take a finite subcover and define $$Y=(y_1,\ldots,y_{p(\varepsilon)})$$
as the sequence of centers of the finite subcover together with the endpoints of the sets $D_i$ in $[x,b]$ all indexed by the increasing order of inclusion of its hitting times in $\mathcal{C}$. Note that $h(y_{i},[y_{i},y_{i+1}])<\varepsilon$.

We construct a new chain as follows: Start with the part $[D_1,D_2]=[\{x\},D_2] \subseteq \mathcal{C}$ which has a parametrization $\Gamma_{1,2}:[0,1] \to \cont(X)$. For the next part, let $D_j$ be the last set for which
$$D_2 \cap [x,b] = D_j \cap [x,b].$$
If $D_j =D_2$, then add $[D_2,D_3] \subseteq \mathcal{C}$ to the chain. If not, there is a parametrization $\gamma_{2,j}:[0,1] \to [D_2,D_j]$ and we define $\Gamma_{2,j}:[0,1] \to \cont(X)$ as
$$\Gamma_{2,j}(t) = \gamma_{2,j}(t) \cup [y_{i_j},y_{i_j+1}](t)$$
where $y_{i_j}$ is the last entry in the finite sequence $Y$ realizing the endpoint of $D_j$ in $[x,b]$, so we know that $h(y_{i_j},[y_{i_j},y_{i_j+1}])<\varepsilon$.
Then, we use a parametrization of $[D_j,D_{j+1}] \subset \mathcal{C}$ to define
$$\Gamma_{j,j+1}(t) = \Gamma_{2,j}(1) \cup \gamma_{j,j+1}(t).$$
Note that from the definition of $y_{i_j+1}$, we have $y_{i_j+1} \in D_{j+1}$, so we have built a chain that starts from $D_2$ and ends at $D_{j+1}$. Moreover, given $t_0 \in [0,1]$,
$$h(\Gamma_{2,j}(t_0),\gamma_{2,j}(t_0))\leq h(\gamma_{2,j}(t_0) \cup [y_{i_j},y_{i_j+1}],\gamma_{2,j}(t_0)) <\varepsilon$$
and
$$h(\Gamma_{j,j+1}(t_0),\gamma_{j,j+1}(t_0))=h(\gamma_{j,j+1}(t_0)\cup [y_{i_j},y_{i_j+1}],\gamma_{j,j+1}(t_0))<\varepsilon.$$
For simplicity, we will use $\ast$ to denote the concatenation of arcs. Therefore, define $\Gamma_{2,j+1}=\Gamma_{2,j} \ast \Gamma_{j,j+1}$ and $\Gamma_{1,j+1}=\Gamma_{1,2} \ast \Gamma_{2,j+1}$ is a parametrization of a chain at a distance not greater than $\varepsilon$ from $[\{x\},D_{j+1}] \subset \mathcal{C}$.

Suppose we have constructed such a chain until $D_i$. Then we repeat the same process yielding a chain parametrized by $\Gamma_{i,k}$ for some $k \in \{i+1,\ldots,k(n)\}$. Then, define
$$\Gamma_{1,k} = \Gamma_{1,i} \ast \Gamma_{i,k}.$$
Since $D_i \in \mathcal{C}$, the same argument given before shows that the chain keeps at a distance smaller than $\varepsilon$ from the chain $[\{x\},D_k] \subseteq \mathcal{C}$. Finally, when $D_{k(n)}$ is reached, add $[D_{k(n)},W_M] \subseteq \mathcal{C}$ to the chain. This yields a maximal order arc $\mathcal{C}'$ such that $h^2(\mathcal{C},\mathcal{C}')<\varepsilon$. We claim $\mathcal{C}' \not\in \scr{R}(n,b)$.

The construction of the chain is done together with a construction of a parametrization $\Gamma$, which is done by juxtaposing three kinds of chains:
\begin{enumerate}[label={(\roman*)}]
 \item $[D_j,D_{j+1}] \subseteq \mathcal{C}$.
 \item $[D_j \cup [y_{i_j},y_{i_j +1}],D_{j+1}]$.
 \item $[D_i,D_j](t) \cup [y_{i_j},y_{i_j+1}](t)$.
\end{enumerate}
To show that $\mathcal{C}' \not\in \scr{R}(n,b)$, suppose that there exist $K_1,K_2 \in \mathcal{C}'$ such that $K_1 \subsetneq K_2$, $h(K_1,K_2)\geq 1/n$, $B(b,1/n) \cap K_1 = \emptyset$ and
$$K_1 \cap [x,b] = K_2 \cap [x,b].$$
There is no $j$ for which $K_1,K_2$ lie in the interval of the cases (i) or (ii) since their endpoints are at a distance smaller than $1/n$. Also, in case (iii), every element of the interval has a different endpoint in $[x,b]$, so $K_1$ and $K_2$ cannot both belong to a single chain of type (iii). Note, however, that if $K_{i_1}$ belongs to the interior of a chain of type (iii), this means that $K_{i_2}$ necessarily have a different endpoint in $[x,b]$, so we are left with two possibilities to analyze: First, $K_1 \in [D_{j_1},D_{{j_1}+1}]$ for some $j_1$ and  $K_2 \in [D_{j_2} \cup [y_{i_{j_2}},y_{i_{j_2} +1}],D_{{j_2}+1}]$ for some $j_2$. This means that there exists an interval of type (iii) between $K_1$ and $K_2$, which again makes the endpoints different. On the other hand, suppose $K_1 \in [D_{j_1} \cup [y_{i_{j_1}},y_{i_{j_1} +1}],D_{{j_1}+1}]$ for some $j_1$ and $K_2 \in [D_{j_2},D_{{j_2}+1}]$ for some $j_2$. It cannot happen that $j_1+1=j_2$, because it means
$$h(D_{j_1},D_{j_2+1})=h(D_{j_1},D_{j_1+2})<1/n$$
and $h(K_1,K_2) \geq 1/n$. Thus, there is an interval of type (i) or (iii) between $K_1$ and $K_2$, hence the endpoints in $[x,b]$ are distinct. Therefore,
$$K_1 \cap [x,b] \ne K_2 \cap [x,b]$$
and $\mathcal{C}' \not\in \scr{R}(n,b)$.

\underline{Homeomorphism construction:}

Let $\mathcal{C}_1$ and $\mathcal{C}_2$ be chains in $\moa(W_M)$ satisfying (i)-(iv). Given $b \in \mathcal{B}(W_M)$, we will denote its hitting time in $\mathcal{C}_1$ by $K_b^1$ and in $\mathcal{C}_2$ by $K_b^2$. First, we will construct a bijection $\varphi:\mathcal{B}(W_M) \to \mathcal{B}(W_M)$ preserving betweenness, such that
 $$\varphi(\mathcal{B}(W_M) \cap K_b^1) = \mathcal{B}(W_M) \cap K_{\varphi(b)}^2$$
 for each $b \in \mathcal{B}(W_M)$.

 Let $x = \operatorname{Root}(C_1)$ and $y=\operatorname{Root}(C_2)$. We fix an enumeration $\{b_n\}$ of $\mathcal{B}(W_M)$ and define inductively functions $\varphi_n : X_n \to \mathcal{B}(W_M)$ and $\psi_n: Y_n \to \mathcal{B}(W_M)$, where
 \begin{itemize}
 \item $\{b_1,\ldots,b_n\} \cup \psi_{n-1}(Y_{n-1}) \subseteq X_n \subset \mathcal{B}(W_M)$ and $\{b_1,\ldots,b_n\} \cup \varphi_n(X_n) \subseteq Y_n \subset \mathcal{B}(W_M)$.
 \item $\varphi_{n-1} \subseteq \varphi_n$ and $\psi_{n-1} \subseteq \psi_n$.
\item $\varphi_n \circ \psi_{n-1} = \id_{Y_{n-1}}$ and $\psi_n \circ \varphi_n=\id_{X_n}$.
 \item $Y_{n-1} \cup \{y\}$ defines a tree $T_{n-1,n-1}^2$ and $\psi_{n-1}(Y_{n-1})\cup \{x\}$ defines a tree $T_{n-1,n-1}^1$ on $W_M$. The domain of $\psi_{n-1}$ is such that if $x_1,x_2 \in Y_{n-1}$, then the endpoint $z$ of $[x_1,y] \cap [x_2,y]$ other than $y$ is also in $Y_{n-1}$ and $\psi_{n-1}(z)=w$ where $w$ is the endpoint of $[\psi_{n-1}(x_1),x] \cap [\psi_{n-1}(x_2),x]$ other than $x$.
 \item $X_{n} \cup \{x\}$ defines a tree $T_{n,n-1}^1$ and $\varphi_{n}(X_{n})\cup \{y\}$ defines a tree $T_{n,n-1}^2$ on $W_M$. The domain of $\varphi_{n}$ is such that if $x_1,x_2 \in X_{n}$, then the endpoint $z$ of $[x_1,x] \cap [x_2,x]$ other than $x$ is also in $X_{n}$ and $\varphi_{n}(z)=w$ where $w$ is the endpoint of $[\varphi_{n}(x_1),y] \cap [\varphi_{n}(x_2),y]$ other than $y$.
 \item Betweenness is preserved by $\varphi_{n}$ and $\psi_{n}$.
 \item $\varphi_n(b') \in K_{\varphi_n(b)}^2$ if and only if $b' \in K_b^1$.
 \item $\psi_n(b') \in K_{\psi_n(b)}^1$ if and only if $b' \in K_b^2$.
 \end{itemize}
 Then, we will have
 $$\varphi=\bigcup_{n \in \bb{N}} \varphi_n \quad \text{and} \quad \psi=\bigcup_{n \in \bb{N}} \psi_n$$
 so that both are functions defined on $\mathcal{B}(W_M)$ and $\psi = \varphi^{-1}$ from the construction.

 First, let $X_1=\{b_1\}$. Define $\varphi_1(b_1)=b_1$ and hence $\psi_1(b_1)=b_1$ so $Y_1=\{b_1\}$. Now suppose that $\varphi_{n-1}$ and $\psi_{n-1}$ are defined. We will define $\varphi_n(b_n)$. First of all, either $b_n \in T_{n-1,n-1}^1$ or not. \underline{\textbf{If} $b_n$ \textbf{is in} $T_{n-1,n-1}^1$}, we define $X_n=\{b_1,\ldots,b_n\} \cup \psi_{n-1}(Y_{n-1})$, define $\varphi_n|_{X_{n-1}}=\varphi_{n-1}$ and $\varphi_n|_{\psi_{n-1}(Y_{n-1})}$ such that $\varphi_n \circ \psi_{n-1} = \operatorname{id}_{Y_{n-1}}$. To define $\varphi_n(b_n)$ we split in cases:
\begin{itemize}
 \item $b_n \in X_{n-1}$: Nothing has to be done.

 \item $b_n \in \psi_{n-1}(Y_{n-1}) \setminus X_{n-1}:$ We have that $\varphi_n(b_n)$ is $b_m \in Y_{n-1}$ for which $\psi_{n-1}(b_m)=b_n$. Hence, betweeness is preserved since $\psi_{n-1}$ preserves betweeness.

 \item $b_n \not\in \psi_{n-1}(Y_{n-1})$: This means $b_n$ is not a node of $T_{n-1,n-1}^1$, hence it lies in the minimal interval $(b_{n_1},b_{n_2})$ defined by nodes of $T_{n-1,n-1}^1$ where $K_{b_{n_1}}^1 \subseteq K_{b_{n_2}}^1$. By the definition of chain, we can find $b_{m_1},b_{m_2} \in X_{n-1} \cup \psi_{n-1}(Y_{n-1})$ such that $K_{b_{m_1}}^1$ is maximal and $K_{b_{m_2}}^1$ is minimal satisfying
  $$K_{b_{m_1}}^1 \subseteq K_{b_n}^1 \subseteq K_{b_{m_2}}^1.$$

 Since $b_n \in (b_{n_1},b_{n_2}) \subset K_{b_{n_2}}^1$ we have that $K_{b_{n}}^1 \subsetneq K_{b_{n_2}}^1$. Thus, we claim that $K_{b_{n_1}}^1 \subseteq K_{b_n}^1$. Suppose towards a contradiction that $K_{b_n}^1 \subsetneq K_{b_{n_1}}^1.$ Then, let $Y=[x,b_{n_1}] \cup [b_{n_1},b_{n_2}]$ and
 $$z=r_{Y,[b_{n_1},b_{n_2}]}(x).$$
 From the fact that $W_M$ is uniquely arcwise connected, we have $[x,b_{n_1}]=[x,z] \cup [z,b_{n_1}]$ and $[x,b_{n_2}] = [x,z] \cup [z,b_{n_2}]$. Moreover, $z \not\in \{b_{n_1},b_{n_2}\}$ since we are supposing $K_{b_n}^1 \subsetneq K_{b_{n_1}}^1$. Thus, by induction hypothesis, $z$ has to be a node of $T_{n-1,n-1}^1$, which is a contradiction since $(b_{n_1},b_{n_2})$ does not contain any node of $T_{n-1,n-1}^1$. 

 Therefore, we have
 $$K_{b_{n_1}}^1 \subseteq K_{b_{m_1}}^1 \subsetneq K_{b_{n}}^1 \subsetneq K_{b_{m_2}}^1 \subseteq K_{b_{n_2}}^1.$$
 Let $z_1$ be the endpoint of $K_{\varphi_n(b_{m_1})}^2$ in $(\varphi_n(b_{n_1}),\varphi_n(b_{n_2}))$ and $z_2$ of $K_{\varphi_n(b_{m_2})}^2$. Choose any
 $$z \in (z_1,z_2) \cap \mathcal{B}(W_M)$$
 having the same order as $b_n$ so that
 $$K_{\varphi_n(b_{m_1})}^2 \subsetneq K_{z}^2 \subsetneq K_{\varphi_n(b_{m_2})}^2.$$
 Let $z=\varphi_n(b_{n})$. This choice is possible since $(\varphi_n(b_{n_1}),\varphi_n(b_{n_2}))$ is an arc $A$ such that $K_{\varphi_n(b_{n_1})}^2 \subseteq K_{\varphi_n(b_{n_2})}^2$ in $\mathcal{C}_2$ and $\emptyset \ne K_{\varphi_n(b_{n_1})}^2 \cap A \subsetneq K_{\varphi_n(b_{n_2})}^2 \cap A \subsetneq A$.

 We claim that the construction preserve betweenness among the points already in the domain and the point added. Given $b_{k_1}$ and $b_{k_2}$ already in $X_n\setminus\{b_n\}$, we can assume either $b_n \in [b_{k_1},b_{k_2}]$ or $b_{k_1} \in [b_n,b_{k_2}]$. In the first case we have $b_n \in [b_{n_1},b_{n_2}]$ and by definition $b_{n_1},b_{n_2} \in [b_{k_1},b_{k_2}]$, so $\varphi_n(b_{n_1}),\varphi_n(b_{n_2}) \in [\varphi_n(b_{k_1}),\varphi_n(b_{k_2})]$ by hypothesis and betweenness is preserved. In the second case we have $b_{k_1} \in [b_n,b_{k_2}]$ so that
 $$b_n \in [b_{n_1},b_{n_2}] \subseteq [b_{n_1},b_{k_1}] \subseteq [b_{n_1},b_{k_2}].$$
 Since all these points $b_{n_1},b_{n_2},b_{k_1},b_{k_2}$ are already in the domain of $\varphi_n$, their betweenness relation is preserved, so $\varphi_n(b_{k_1}) \in [\varphi_n(b_{n_2}),\varphi_n(b_{k_2})]$ and hence $\varphi_n(b_{k_1}) \in [\varphi_n(b_{n}),\varphi_n(b_{k_2})]$.
 \end{itemize}

 {\bf{\underline{If $b_n \not\in T_{n-1,n-1}^1$}}}, then there exist optimal $b_{m_1},b_{m_2}$ in $\psi_{n-1}(Y_{n-1})$ for which one of the following holds
 \begin{enumerate}[label={(\roman*)}]
 \item $K_{b_{m_1}}^1 \subsetneq K_{b_{n}}^1 \subsetneq K_{b_{m_2}}^1$.
 \item $\{x\} \subsetneq K_{b_{n}}^1 \subsetneq K_{b_{m_2}}^1$.
 \item $K_{b_{m_1}}^1 \subsetneq K_{b_{n}}^1 \subsetneq W_M$.
 \end{enumerate}
 By optimal we mean that $K_{b_{m_1}}^1$ is maximal satisfying $K_{b_{m_1}}^1 \subseteq K_{b_n}^1$ and $K_{b_{m_2}}^1$ is minimal satisfying $K_{b_n}^1 \subseteq K_{b_{m_2}}^1$.
 In any case, given that
 $$z = r_{W_M,T_{n-1,n-1}^1}(b_n) \in T_{n-1,n-1}^1,$$
 we can define $X_n=\{b_1,\ldots,b_n\} \cup \psi_{n-1}(Y_{n-1}) \cup \{z\}$ and define $\varphi_n(z)$ as in the case above. Let
 $$Y = W_M \setminus \cup \{K \in \mathcal{C}_2 \ | \ K \subsetneq W_M\}.$$
 Note that $Y$ is dense so we can fix
 $$z' \in r^{-1}_{W_M,K_{\varphi_{n}(z)}^2}(z) \cap Y.$$
 By definition, $[\varphi_n(z),z']$ is nondegenerate and
 $$\varphi_n(z) \in K_{\varphi_n(z)}^2 \subsetneq K_{\varphi_n(b_{m_2})}^2$$
 so in case (i) we can call $z_2$ the endpoint of $K_{\varphi_n(b_{m_2})}^2$ in $[\varphi_n(z),z']$. Therefore, we define $\varphi_n(b_n)$ as any branching point $b \in (\varphi_n(z),z_2) \setminus K_{\varphi_n(b_{m_1})}^2$ having the same order as $b_n$. The choice of $b$ is possible since willfulness implies that the set is nonempty. In case (ii) it is enough to take any element in $(\varphi_n(z),z_2)$ and in case (iii) it is not necessary to fix $z_2$, so the choice is made on $(\varphi_n(z),z') \setminus K_{\varphi_n(b_{m_1})}^2$.

 For betweenness, again consider $b_{k_1},b_{k_2}$ in $X_n$, but in this case only $b_{k_1} \in [b_n,b_{k_2}]$ is possible, since $[b_{k_1},b_{k_2}] \subseteq T_{n-1,n-1}^1$ and $b_n \not\in T_{n-1,n-1}^1$ so $b_n \not\in [b_{k_1},b_{k_2}]$. We know that
 $$z = r_{W_M,T_{n-1,n-1}^1}(b_n) \in T_{n-1,n-1}^1$$
 was added to the domain and by construction $b_{k_1} \in [z,b_{k_2}]$, so $\varphi_n(b_{k_1}) \in [\varphi_n(z),\varphi_n(b_{k_2})]$ from previous case and
 $$\varphi_n(b_{k_1}) \in [\varphi_n(b_n),\varphi_n(b_{k_2})].$$

 Finally, we need to preserve the induction hypothesis, but the construction makes $[x,b_n] \cap T_{n,n-1}^1 = [x,z]$ and thus either the endpoint other than $x$ in
 $$[b_m,x]\cap [z,x]$$
 is already a node of $T^1_{n-1,n-1}$ or it is $z$, which is a node now. The same holds for $T^2_{n,n-1}$.

 An analogous argument is used to define $Y_n$ and $\psi_{n}(Y_n)$ from $T^2_{n,n-1}$, yielding trees $T^1_{n,n}$ and $T^2_{n,n}$. Then, we fall again on the inductive case, which means we can define the function $\varphi:\mathcal{B}(W_M) \to \mathcal{B}(W_M)$ that preserves betweenness relation and also the hitting levels and the order.

From Proposition \ref{duchesnehomeo} this bijection $\varphi$ can be extended to a homeomorphism $\Phi: W_M \to W_M$. By construction $\varphi$ is such that $K_{b_{n_1}}^1 \subseteq K_{b_{n_2}}^2$ implies $K_{\varphi(b_{n_1})}^2 \subseteq K_{\varphi(b_{n_2})}^2$. Also, since every hitting time is homeomorphic to $W_M$, we have that for any $b \in \mathcal{B}(W_M)$
  $$\overline{\mathcal{B}(K_{b}^1)} = K_{b}^1 \quad \text{and} \quad \overline{\mathcal{B}(K_{\varphi(b)}^2)} = K_{\varphi(b)}^2.$$
  Moreover, $\overline{\varphi(\mathcal{B}(K_{b}^1))} = \overline{\mathcal{B}(K_{\varphi(b)}^2)}$ since $\varphi^{-1}$ has the same properties of $\varphi$.
  Therefore, it holds that
  $$\Phi(K_b^1)=\Phi(\overline{\mathcal{B}(K_{b}^1)}) = \overline{\Phi(\mathcal{B}(K_{b}^1))} = \overline{\varphi(\mathcal{B}(K_{b}^1))} = \overline{\mathcal{B}(K_{\varphi(b)}^2)}=K_{\varphi(b)}^2.$$
  Since $\mathcal{C}_i$ is willful, the collection of hitting times of branching points is dense in both chains. Therefore, the homeomorphism $\overline{\Phi}: \cont(W_M) \to \cont(W_M)$ induced by $\Phi$ as $\overline{\Phi}(K)=\Phi(K)$ is such that $\overline{\Phi}(\mathcal{C}_1)=\mathcal{C}_2$.
  
\end{proof}

\subsection{Another class of equivalent chains of $\moa(W_\omega)$}

On \cite{banicinverse}, the Wa\.zewski universal dendrite $W_\omega$ is constructed as the generalized inverse limit of a single set-valued bonding function $f$. We can use this construction to create a maximal order arc in $\moa(W_\omega)$.

The bonding function (or a relation) $f:[0,1] \to \mathcal{P}([0,1])$ used to describe the Wa\.zewski dendrite is defined by its graph as follows: let $\{(a_n,b_n)\}_{n \in \bb{N}}$ be a collection of pairs where $\mathcal{A}=\{a_n\}_{n \in \bb{N}}$ is a dense collection of points in $[0,1]$, $a_n < b_n$, and for every $n$ there exists $m_n \in \bb{N}$ for which $b_m < a_n$ whenever $a_m < a_n$ and $m \geq m_n$. Then $f$ is the set-valued function such that for every $s \in [0,1]$
$$f^{-1}(s)= \begin{cases}
    [a_n,b_n], \text{ if } s=a_n \text{ for some } n \in \bb{N},\\
    \{s\}, \text{ otherwise.}
\end{cases}$$
This means that the graph of $f$ is the diagonal with an arc starting at $(a_n,a_n)$ and ending at $(b_n,a_n)$ attached to the diagonal for every $a_n \in \mathcal{A}$. Then
$$W_\omega = \varprojlim \{[0,1],f\}_{k=1}^\infty = \{x \in [0,1]^{\bb{N}} \ | \ x_i \in f(x_{i+1}) \text{ for each } i \in \bb{N}\}.$$

For every $t \in [0,1]$ we can define a new function $f_t:[0,t] \to \mathcal{P}([0,t])$ as
$$f_t^{-1}(s)= \begin{cases}
    [a_n,a_n+(t-a_n)(b_n-a_n)/(1-a_n)], \text{ if } s=a_n \text{ for some } n \in \bb{N},\\
    \{s\}, \text{ otherwise.}
\end{cases}$$
By definition, the inverse limit defined using $f_t$ also yields the Wa\.zewski dendrite as a result, but it can be seen as a subspace of
$$\varprojlim \{[0,1],f\}_{k=1}^\infty.$$
Hence, we can create a map $\gamma:[0,1] \to \cont(W_M)$ given by
$$\gamma(t) = \varprojlim \{[0,t],f_t\}_{k=1}^\infty$$
that is a maximal order arc in $\cont(W_M)$.

In this construction, sequences of the form $(x_1,\ldots,x_k,a_n,a_n,\ldots)$ for  some $a_n \in \mathcal{A}$ represent branching points of $W_\omega$. On the other hand, if $f_t^{-1}(a_n)=[a_n,c]$, a sequence of the form $(x_1,\ldots,x_k,a_n,c,c,\ldots)$ represent endpoints of $\gamma(t)$. Therefore, if  $\mathcal{A}=\bb{Q} \cap [0,1]$, then $f_t^{-1}(a_n)$ is an interval with rational endpoints for every $n \in \bb{N}$. Thus, $f_t^{-1}(a_n)=[a_n,a_m]$ for some $m \in \bb{N}$ and in $\gamma(t)$ we can find points of the form
$$(x_1,\ldots,x_k,a_n,a_m,a_m,\ldots)$$
which are endpoints of $\gamma(t)$ but also branching points of $W_\omega$. Moreover, these kind of points are dense, since $\mathbb{Q} \cap [0,1]$ is dense, so we have found a chain having among other properties, the property that its elements have densely many branching points as endpoints. 

Note that the chain described above has the opposite behavior when compared to the ones in the comeager set described in Theorem \ref{ambientlyeqWM}, however, it is still possible to show that chains hitting densely many branching points at  time with the other same properties as the chains in the comeager set are ambiently equivalent.

\begin{theo}\label{ambientlyeqW}
There exists a chain $\mathcal{C} \in \moa(W_\omega)$ satisfying the properties below:
\begin{enumerate}[label={(\roman*)}]
   \item The root of $\mathcal{C}$ is an endpoint of $W_\omega$.
   \item If $K,L \in \mathcal{C}$ with $K \subsetneq L$, then $K$ is nowhere dense in $L$.
   \item If $K \in \mathcal{C}$ and $\ep(K) \cap \mathcal{B}(W_\omega) \ne \emptyset$, then
   $$\ep(K) \subseteq \mathcal{B}(W_\omega) \cup \ep(W_\omega)$$
   and $\ep(K) \cap \mathcal{B}(W_\omega)$ is dense in $K$.
   \item $\mathcal{C}$ is willful.
  \end{enumerate}
\end{theo}
Moreover, any two chains of $\moa(W_\omega)$ satisfying the conditions above are ambiently equivalent.
\begin{proof}

The existence was discussed in the beginning of this section. Given $y \in W_\omega,$ we denote by $K_y^i$ the hitting time of $y$ in $\mathcal{C}_i$. The construction of the homeomorphism between two chains $\mathcal{C}_1$ and $\mathcal{C}_2$ satisfying these properties is done in a very similar to the proof of Theorem \ref{ambientlyeqWM}. However, since the hitting levels may contain more than one branching point, it may happen $K_{b_{m_1}}^1=K_{b_n}^1$ or $K_{b_{m_2}}^1=K_{b_n}^1$. When $b_n \in T_{n-1,n-1}^1$, $b_n$ is also an endpoint of $K_{b_{m_1}}^1$ that lies on $(b_{n_1},b_{n_2})$. Therefore, we choose as $\varphi_n(b_n)$ the endpoint of $K_{\varphi_n(b_{m_1})}^2$ in $(\varphi_n(b_{n_1}),\varphi_n(b_{n_2}))$ and
 $$K_{\varphi_n(b_{m_1})}^2 = K_{\varphi_n(b_{n})}^2.$$
 Betweenness is preserved by the minimality of the interval $(b_{n_1},b_{n_2})$. If $b_{n} \not\in T_{n-1,n-1}^1$, let again
$$z = r_{W_\omega,T_{n-1,n-1}^1}(b_n) \in T_{n-1,n-1}^1,$$
 hence we define $\varphi_n(z)$ as in the case when $b_n \in T_{n-1,n-1}^1$. Note that $\varphi_n(z) \not\in \ep(K_{\varphi_n(b_{m_1})}^1)$ since $b_n \not\in T_{n-1,n-1}^1$ and $K_{b_{m_1}}^1=K_{b_n}^1$. Thus,
 $$(r^{-1}_{W_\omega,T_{n-1,n-1}^2}(\varphi_n(z)) \setminus T_{n-1,n-1}^2)$$
 is open in $W_\omega$ and meets $K_{\varphi_n(b_{m_1})}^2$. Hence, there is an open component of $r^{-1}_{W_\omega,T_{n-1,n-1}^2}(\varphi_n(z))$ that does not meet $T_{n-1,n-1}^2$, and there we can find
 $$z' \in  (r^{-1}_{W_\omega,K_{\varphi_n(z)}^2}(\varphi_n(z))  \cap (T(\varphi_n(b_{m_1}),\mathcal{C}_2) \cap \mathcal{B}(W_\omega))$$
 because $\ep(K_{\varphi_n(b_{m_1})})$ is dense in $K_{\varphi_n(b_{m_1})}$. Therefore, we define $z' = \varphi_n(b_n)$. Betweenness is preserved by the same argument as before.

 The construction of the homeomorphism is also done in the same fashion.
\end{proof}

\printbibliography
\end{document}